\documentclass[12pt]{article}
\usepackage[centertags]{amsmath}
\usepackage{times,fullpage}
\usepackage{enumerate}
\usepackage{latexsym,amscd,amssymb}
 \usepackage[palatino,gill,courier]{altfont}
 \usepackage{mathrsfs}
\usepackage{enumitem}
\usepackage{amsmath, amsfonts, amssymb, amsthm, amscd}
\usepackage[all]{xy}
\usepackage{newlfont}
\usepackage[margin=1in]{geometry}
\usepackage[colorlinks,linkcolor=blue, pdfstartview=FitH]{hyperref}
\usepackage{lipsum}
\usepackage{enumerate}
\numberwithin{equation}{section}

 \theoremstyle{plain}
\newtheorem{thm}{Theorem}[section]
\newtheorem{theorem}[thm]{Theorem}
\newtheorem{lemma}[thm]{Lemma}
\newtheorem{corollary}[thm]{Corollary}
\newtheorem{proposition}[thm]{Proposition}

\theoremstyle{definition}

\newtheorem{remark}[thm]{Remark}

\newtheorem{definition}[thm]{Definition}
\newtheorem{claim}[thm]{Claim}

\newtheorem{example}[thm]{Example}

\newtheorem{defn-thm}[thm]{Definition-Theorem}

\renewcommand{\H}{{\mathbb H}}

\newcommand{\tr}{{ tr}}

\newcommand{\bp}{\bar{\partial}}
\newcommand{\Om}{\Omega}

\newcommand{\btheorem}{\begin{theorem}}
\newcommand{\etheorem}{\end{theorem}}
\newcommand{\bproposition}{\begin{proposition}}
\newcommand{\eproposition}{\end{proposition}}
\newcommand{\bdefinition}{\begin{definition}}
\newcommand{\edefinition}{\end{definition}}
\newcommand{\bcorollary}{\begin{corollary}}
\newcommand{\ecorollary}{\end{corollary}}
\newcommand{\bproof}{\begin{proof}}
\newcommand{\eproof}{\end{proof}}
\newcommand{\bremark}{\begin{remark}}
\newcommand{\eremark}{\end{remark}}
\newcommand{\eexample}{\end{example}}
\newcommand{\bexample}{\begin{example}}

\newcommand{\elemma}{\end{lemma}}
\newcommand{\blemma}{\begin{lemma}}

\newcommand{\suml}{\sum\limits}

\newcommand{\p}{\partial}

\renewcommand{\bar}{\overline}

\renewcommand{\phi}{\varphi}

\newcommand{\ee}{\end{eqnarray*}}
\newcommand{\be}{\begin{eqnarray*}}

\newcommand{\beq}{\begin{equation}}
\newcommand{\eeq}{\end{equation}}

\newcommand{\bd}{\begin{enumerate} [(1)]}
\newcommand{\ed}{\end{enumerate}}

\renewcommand{\tilde}{\widetilde}

\renewcommand{\bf}{\textbf}

\renewcommand{\it}{\textit}

\renewcommand{\bf}{\textbf}

\begin{document}
\title{Quantum Correction and the Moduli Spaces of Calabi-Yau Manifolds}
\author{Kefeng Liu, Changyong Yin}

\AtEndDocument{\bigskip{\footnotesize%
\textsc{Center of Mathematical Sciences, Zhejiang University, Hangzhou, Zhejiang 310027, China;
        Department of Mathematics,University of California at Los Angeles,
        Los Angeles, CA 90095-1555, USA} \par
\textit{E-mail address:} \texttt{liu@math.ucla.edu, liu@cms.zju.edu.cn} \par
~\\~
\noindent \textsc{Department of Mathematics,University of California at Los Angeles,
        Los Angeles, CA 90095-1555, USA} \par
\textit{E-mail address:} \texttt{changyong@math.ucla.edu}
}}

\date{}
 \maketitle
\begin{abstract}
\noindent We define the quantum correction of the Teichm\"uller space $\mathcal{T}$ of Calabi-Yau manifolds.
Under the assumption of no weak quantum correction, we prove that the Teichm\"uller space $\mathcal{T}$ is a locally
symmetric space with the Weil-Petersson metric. For Calabi-Yau threefolds, we show that no strong quantum correction
is equivalent to that, with the Hodge metric, the image $ \Phi(\mathcal{T})$ of the Teichm\"uller space $\mathcal{T}$ 
under the period map $\Phi$ is an open submanifold of
a globally Hermitian symmetric space $W$ of the same dimension as $\mathcal{T}$. Finally, for Hyperk\"ahler manifold of dimension 
$2n \geq 4$, we find both locally and globally defined families of $(2,0)$ and $(2n,0)$-classes over the Teichm\"uller space
of polarized Hyperk\"ahler manifolds. 
\end{abstract}

\tableofcontents
\section{Introduction}

In this paper we study the Teichm\"uller space of Calabi-Yau and Hyperk\"ahler manifolds.
Recall that a compact projective manifold $X$ of complex dimension $n$ with $n
\geq 3$ is called a Calabi-Yau manifold, if it has a trivial
canonical bundle and satisfies $H^i(X, \mathcal{O}_X)=0$ for $0 < i < n$.
A compact and simply connected projective manifold of complex dimension $2n\geq4$
is called Hyperk\"ahler, if it admits a nondegenerate holomorphic $(2,0)$-form.

\noindent In the first part, we study Calabi-Yau manifolds.
A polarized and marked Calabi-Yau manifold
is a triple $(X, L, \gamma)$ of a Calabi-Yau manifold $X$, an ample line bundle $L$ over $X$ and
a basis $\gamma$ of the integral middle homology group modulo torsion, $H_n(X, \mathbb{Z})/ \text{Tor}$.
We will denote the moduli space of a polarized and marked Calabi-Yau manifold $X$ by $\mathcal{T}$.
Actually, the Teichm\"uller
space $\mathcal{T}$ is precisely the universal cover of the smooth moduli space $\mathcal{Z}_m$ of polarized
Calabi-Yau manifolds with level $m $ structure with $m \geq 3$, which is constructed by Popp, Viehweg and Szendr\"oi,
for example in \cite[Section 2]{S}. Here, A basis of the quotient space
$(H_n(X,\mathbb{Z})/\text{Tor})/m(H_n(X,\mathbb{Z})/\text{Tor})$ is called a level $m$
structure with $m \geq 3$ on the polarized Calabi--Yau manifold $(X,L)$.
The versal family $\mathcal{X} \rightarrow \mathcal{T}$ of
polarized and marked Calabi-Yau manifolds is the pull-back of the versal family $\mathcal{X}_{\mathcal{Z}_m}
\rightarrow \mathcal{Z}_m$, see \cite{S}. Therefore, $\mathcal{T}$  is a connected and
simply connected smooth complex manifold.

\noindent  There has been very active studies about when a horizontal subvariety of a Griffiths period domain is a globally Herimitian
symmmetric space and to classify the possible variations of Hodge structure parameterized by a globally
Hermitian symmetric space. The reader can refer \cite{FL} and its reference.
In this paper, we will  characterize a class of Calabi-Yau manifolds whose
Teichm\"uller spaces are locally Hermitian symmetric spaces or
globally Herimitian symmetric spaces, by studying the quantum correction as motivated by mirror symmetry from String Theory.

\noindent Fix $p \in \mathcal{T}$, let $X$ be the corresponding Calabi-Yau manifold in the versal family and
 $\{\varphi_1$, $\cdots$, $\varphi_N \} \in \mathbb{H}^{0,1}(X, T^{1,0}X)$ be an orthonormal basis of harmonic Beltrami differentials with respect to the
Calabi-Yau metric. Then we can construct a smooth family of Betrami differentials
\begin{eqnarray*}
\Phi(t)=\sum_{|I|\geq 1} t^I \varphi_I,
\end{eqnarray*}
which describe the deformation of complex structures around $p \in \mathcal{T}$.
Our essential idea is to consider the strong quantum correction
at point $p \in \mathcal{T}$, which can be simply described as the following identity of cohomology classes,
\begin{eqnarray*}
[\Xi(t)]= [\Omega^c(t)]-[\exp(\sum_{i=1}^N t_i \phi_i)\lrcorner \Omega],
\end{eqnarray*}
where $\Omega$ is a holomorphic $(n,0)$-form over $X$ and $\Omega^c(t)$ is the
canonical family of holomorphic $(n,0)$-forms in a neighborhood of $p \in \mathcal{T}$.
And the weak quantum correction at point $p \in \mathcal{T}$ is defined as
the lowest order expansion of $[\Xi(t)]$ with respect to $T$, i.e.,
\begin{eqnarray*}
[\Xi(t)]_1= \sum_{i,j,k=1}^N t_i t_j t_k [\phi_i \lrcorner \phi_{jk} \lrcorner \Omega].
\end{eqnarray*}
\noindent  When $n=3$, i.e.,
for Calabi-Yau threefolds, no strong quantum correction implies
no quantum correction of the Yukawa coupling
in physics literatures, the reader can refer to \cite{BCOV} for details about
quantum correction of the Yukawa coupling.

\noindent We first prove the following result which characterizes the Teichm\"uller space $\mathcal{T}$
of Calabi-Yau manifolds without quantum correction,
\begin{theorem}
Let $\mathcal{T}$ be the Teichm\"uller space of polarized and marked Calabi-Yau manifolds.
If there is no weak quantum correction
at any point $p \in \mathcal{T}$, i.e., $[\Xi(t)]_1=0$,
 then $\mathcal{T}$
is a locally Hermitian symmetric space with the Weil-Petersson metric.
\end{theorem}

\noindent Here, locally Hermitian symmetric property is equivalent to $\nabla R=0$,
which is not necessarily complete. Moreover, for polarized and marked Calabi-Yau threefolds, we found the
following equivalent condition for no strong quantum correction
at any point $p \in \mathcal{T}$.

\begin{theorem}
Let $\mathcal{T}$ be the Teichm\"uller space of polarized and marked Calabi-Yau threefolds
and $\Phi: \mathcal{T} \rightarrow D$ be the period map.
Then the following are equivalent:
\begin{enumerate}
\item $\mathcal{T}$ has no strong quantum correction at any point $p \in \mathcal{T}$;
\item With respect to the Hodge metric,
the image $ \Phi(\mathcal{T})$ is an open submanifold of
a globally Hermitian symmetric space $W$ of the same dimension as $\mathcal{T}$,
which is also a totally geodesic submanifold of
the period domain $D$.
\end{enumerate}
\end{theorem}

\noindent In the second part, we study Hyperk\"ahler manifolds.
Let $\mathcal{T}$ be the Teichm\"uller space of polarized Hyperk\"ahler manifolds.
By explicitly computing the Taylor expansions of the canonical families of  $(2,0)$
and $(2n,0)$-classes at any point $p \in \mathcal{T}$, we show that they have
no strong quantum correction at any point $p \in \mathcal{T}$. Therefore the Teichm\"uller spaces of polarized Hyperk\"ahler manifolds are locally Hermitian symmetric with the Weil-Petersson metric, which is proved without using the local Torelli theorem for Hyperk\"ahler manifolds. We then show that these local expansions are actually global expansions on the Teichm\"uller spaces. More precisely,
let $X$ be a Hyperk\"ahler manifold with $\dim_{\mathbb{C}} X = 2n$ and $\Omega^{2,0}$ be a nowhere
vanishing $(2,0)$-form over $X$, then we can actually construct a globally defined families of $(2,0)$ and $(2n,0)$-classes
over the Teichm\"uller space $\mathcal{T}$.
\begin{theorem}
Fix $p \in \mathcal{T}$, let $X$ be the corresponding Hyperk\"ahler manifold in the versal family and $\Omega^{2,0}$ be a nowhere
vanishing $(2,0)$-form over $X$,
then, in a neighborhood $U$ of $p$, there exists the local families of $(2,0)$ and $(2n,0)$-classes defined
by the canonical families $[\mathbb{H}(e^{\Phi(t)}\lrcorner \Omega^{2,0})]$
and $[e^{\Phi(t)}\lrcorner \wedge^n \Omega^{2,0}]$.
Furthermore their expansions are actually globally defined over the Teichm\"uller space $\mathcal{T}$, i.e.,
\begin{eqnarray*}
&&[\Omega^{2,0}] +
{\displaystyle\sum\limits_{i=1}^{N}}
 [\phi_{i}\lrcorner\Omega^{2,0}] t_{i}+\frac{1}{2}%
{\displaystyle\sum\limits_{i=1}^{N}}
[  \phi_{i}\lrcorner\phi_{j}\lrcorner\Omega
^{2,0}   ] t_{i}t_{j} \in H^{2,0}(X_t),\\
&&\left[  \wedge^n \Omega^{2,0}\right] +
{\displaystyle\sum\limits_{i=1}^{N}}
\left[  \phi_{i}\lrcorner  \wedge^n \Omega^{2,0}\right] t_{i}\\
&&+\frac{1}{k!}
{\displaystyle\sum\limits_{k=1}^{2n}}
\left({\displaystyle\sum\limits_{1\leq i_{1}\leq...\leq i_{k}\leq N}}
[  \phi_{i_{1}}\lrcorner...\lrcorner\phi_{i_{k}}\lrcorner \wedge^n\Omega^{2,0}
]  t_{i_1} t_{i_2} \cdots t_{i_k} \right) \in H^{2n,0}(X_t),
\end{eqnarray*}
are globally defined over $\mathcal{T}$.
\end{theorem}

\noindent This paper is organized as follows. In Section \ref{local}, we briefly
review the construction of the Teichm\"uller space of polarized and marked Calabi-Yau manifolds,
the local deformation theory of Calabi-Yau manifolds and the construction of the canonical family of $(n,0)$-forms.
In Section \ref{symmetric}, the definition and criteria of Hermitian symmetric space are introduced.
Also, we define the quantum correction of the Teichm\"uller space $\mathcal{T}$, which
originally comes from physics literatures.
In Section \ref{qc and wpmetric}, we review the Weil-Petersson metric over the Teichm\"uller space $\mathcal{T}$,
and find a local formula for the covariant derivatives of the curvature tensor $\nabla_r R$ and $\nabla_{\overline{r}} R$ in the flat affine coordinate $t$.
Under the assumption of no weak quantum correction at any point $p \in \mathcal{T}$,
we prove that the Teichm\"uller space $\mathcal{T}$ is a locally Hermitian symmetric space with the
Weil-Petersson metric by computing the local formulas of $\nabla_r R$ and $\nabla_{\overline{r}} R$.
We remark that the results in Sections \ref{local deformation} to \ref{qc and wpmetric} actually
all hold for both Calabi-Yau and Hyperk\"ahler manifolds.
In Section \ref{qc and hodge}, for Calabi-Yau threefolds, we show that
no strong quantum correction is equivalent to that
the image $\Phi(\mathcal{T})$ of the Teichm\"uller space $\mathcal{T}$ under the period map is an open submanifold of
a globally Hermitian symmetric space with the same dimension as $\mathcal{T}$.
In Section \ref{hkmanifold}, we construct a globally defined families of $(2,0)$ and $(2n,0)$-classes
over the Teichm\"uller space $\mathcal{T}$ of polarized Hyperk\"ahler manifolds with $\dim_{\mathbb{C}} X=2n$.
\\
\\
\noindent \textbf{Acknowledgement}:
The authors would like to thank Professors Xiaofeng Sun and Andrey Todorov for useful discussions on some topics in this paper.

\section{Locally Geometric Structure of the Moduli Space}\label{local}
In Section \ref{construction}, we review the construction of
the Teichm\"ulller space of polarized and marked Calabi--Yau manifolds based on the works of
Popp \cite{P}, Viehweg \cite{V} and Szendr\"oi \cite{S}.
In Section \ref{local deformation}
and Section \ref{canonicalfamily}, the smooth family of Betrami
differentials $\Phi(t)$ and the canonical family of $(n,0)$-forms over the
deformation space of Calabi-Yau manifolds was introduced.
The results in Section \ref{local deformation} and Section  \ref{canonicalfamily}
also hold for polarized Hyperk\"ahler manifolds.

\subsection{The Construction of the Teichm\"uller Space}\label{construction}

In this section, we briefly review the construction of the Teichm\"uller space
of polarized and marked Calabi-Yau manifolds and its basic properties.
For the concept of Kuranishi family of compact complex manifolds,
we refer to \cite[Pages $8$-$10$]{SU}, \cite[Page 94]{P} or \cite[Page 19]{V}
for equivalent definitions and more details.
If a complex analytic family
$\pi: \mathcal{X}\rightarrow S$ of compact complex manifolds is complete at each point of $S$ and versal at the point
$0 \in S$, then the family $\pi: \mathcal{X}\rightarrow S$ is called the Kuranishi family of the complex maniflod
$X=\pi^{-1}(0)$. The base space $S$ is called the Kuranishi space. If the family is complete at each point of a
neighbourhood of $0 \in S$ and versal at $0$, then this family is called a local Kuranishi family at
$0 \in S$. In particular, by definition, if the family is versal at each point of $S$, then it is local Kuranishi
at each point of $S$.

\noindent A polarized and marked Calabi-Yau manifold
is a triple of a Calabi-Yau manifold $X$, and an ample line bundle $L$ over $X$ and
a basis of the integral middle homology group modulo torsion, $H_n(X, \mathbb{Z})/ \text{Tor}$.
A basis of the quotient space
$(H_n(X,\mathbb{Z})/\text{Tor})/m(H_n(X,\mathbb{Z})/\text{Tor})$ is called a level $m$
structure on the polarized Calabi--Yau manifold with $m\geq 3$.
For deformation of
a polarized Calabi-Yau manifold with level $m$ structure, we have the following theorem, which is a reformulation
of \cite[Theorem 2.2]{S}. One can also refer \cite{P} and
\cite{V} for more details about the construction of the moduli space of Calabi-Yau manifolds.

\begin{thm}
Let $m \geq 3$ and $ (X, L)$ be a polarized Calabi-Yau manifold with level $m$ structure, then
there exists a quasi-projective complex manifold $Z_m$ with a versal family of Calabi-Yau
maniflods,

\begin{equation}\label{m-family}
\mathcal{X}_{Z_m} \longrightarrow Z_m,
\end{equation}
containing $X$ as a fiber, and polarized by an ample line bundle $\mathcal{L}_{Z_m}$ on the
versal family $\mathcal{X}_{Z_m}$.
\end{thm}

\noindent  Define the Teichm\"uller space $\mathcal{T}_L(X)$ to be the universal
cover of the base space $Z_m$ of the versal family above,
\begin{equation*}
\pi: \mathcal{T}_L(X) \longrightarrow Z_m
\end{equation*}
and the family
\begin{equation*}
\mathcal{X} \longrightarrow \mathcal{T}_L(X)
\end{equation*}
to be the pull-back of the family (\ref{m-family}) by the projection $\pi$. For simplicity, we will denote
$\mathcal{T}_L(X)$ by $\mathcal{T}$. And the Teichm\"uller space has the following property:

\bproposition\label{simplyconnected}
The Teichm\"uller space $\mathcal{T}$ is a simply connected smooth complex manifold, and the family
\begin{equation}
\mathcal{X} \longrightarrow \mathcal{T}
\end{equation}
containing $X$ as a fiber, is local Kuranishi at each point of the Teichm\"uller space $\mathcal{T}$.
\eproposition

\noindent Note that the Teichm\"uller space $\mathcal{T}$ does not depend on the choice of level $m$.
In fact, let $m_1, m_2$ be two different positive integers, $\mathcal{X}_1 \rightarrow \mathcal{T}_1$ and $\mathcal{X}_2
\rightarrow \mathcal{T}_2$ are two versal families constructed via level $m_1$ and level $m_2$ respectively as
above, both of which contain $X$ as a fiber. By using the fact that $\mathcal{T}_1$ and $\mathcal{T}_2$ are
simply connected and the definition of versal families, we have a biholomorphic map $f: \mathcal{T}_1
\rightarrow \mathcal{T}_2$, such that the versal family $\mathcal{X}_1 \rightarrow \mathcal{T}_1$ is the pull-
back of the versal family $\mathcal{X}_2 \rightarrow \mathcal{T}_2$ by the map $f$. Thus these two families are
isomorphic to each other. One can check that the Teichm\"ller space defined above is precisely the moduli space of marked and polarized Calabi-Yau manifolds. See \cite{CGL13} for more details.

\subsection{Local Deformation of Calabi-Yau Manifolds}\label{local deformation}

Fix $p \in \mathcal{T}$, we denote the corresponding polarized Calabi-Yau manifold in the versal family
by $(X, L)$. Yau's solution of Calabi conjecture assigns a unique Calabi-Yau metric
$g$ on $X$, and its imaginary part $\omega= \mbox{Im} \,g \in L$ is the corresponding K\"ahler
form. Under the Calabi-Yau metric $g$, we have the following lemma which follows from the Calabi-Yau theorem directly,
 \blemma \label{isometry}
 Let $\Omega$ be a nowhere vanishing holomorphic $(n,0)$-form on $X$ such that
 \be
 (\frac{\sqrt{-1}}{2})^n (-1)^{\frac{n(n-1)}{2}} \Omega \wedge \overline{\Omega}=\omega^n.
 \ee
 Then the map $\iota : A^{0,1}(X, T^{1,0}X) \rightarrow A^{n-1,1}(X)$ given by
 $\iota(\varphi)=\varphi \lrcorner \Omega$ is an isometry with respect to the natural
 Hermitian inner product on both spaces induced by the Calabi-Yau metric $g$.
 Furthermore, the map $\iota$ preserves the Hodge decomposition.
 \elemma

\noindent Under the Calabi-Yau metric $g$, we have a precise description of the local deformation
of the polarized Calabi-Yau manifolds. By the Hodge theory, we have the following
identification

\be
T^{1,0}_p \mathcal{T} \cong \mathbb{H}^{0,1}(X, T^{1,0}X),
\ee

\noindent where $X$ is the corresponding fiber in the versal family $\mathcal{X} \rightarrow \mathcal{T}$.
By the Kodaira-Spencer-Kuranishi theory,
we have the following convergent power series expansion of the
Betrami differentials, which is now well-known as the Bogomolov-Tian-Todorov Theorem \cite{B78, Ti, To}.
\btheorem \label{BTT}
Let $X$ be a Calabi-Yau manifold and
 $\{\varphi_1$, $\cdots$, $\varphi_N \} \in \mathbb{H}^{0,1}(X, T^{1,0}X)$ be a basis.
 Then for any nontrivial holomorphic $(n,0)$-form $\Omega$ on $X$, we can construct a smooth
 power series of Betrami differentials as follows
\beq \label{betrami}
\Phi(t)=\sum_{|I|\geq 1} t^I \varphi_I =\sum_{\substack{\nu_{1}+\cdots +\nu_{N}\geq1,\\\text{each
$\nu_{i}\geq0, i=1,2,\cdots,N$}}}\varphi_{\nu
_{1}\cdots\nu_{N}}t^{\nu_{1}}_{1}\cdots t^{\nu_{N}}_{N}\in A^{0,1}%
(X,T^{1,0}_X),
\eeq
where $\varphi_{0\cdots\nu_{i}\cdots0}=\varphi_{i}$. This power
series has the following properties:

$1)$ $\overline{\partial}\Phi(t)=\frac{1}{2}[\Phi(t),\Phi(t)]$, the
integrability condition;

$2)$ $\overline{\partial}^{*}\varphi_{I}=0$ for each multi-index $I$
with $|I|\geq1$;

$3)$ $\varphi_{I}\lrcorner\Omega$ is $\partial$-exact for each
$I$ with $|I|\geq2$.

$4)$ it converges when $|t| < \epsilon$.
\etheorem

\noindent For the convergent radius, the reader can refer to \cite[Theorem 4.4]{LRY}.
 This theorem will be used to define the local flat affine coordinates $\{t_1, \cdots, t_N\}$
around the point $p \in \mathcal{T}$, for a given orthonormal basis $\{\varphi_1, \cdots, \varphi_N \}$
of $ \mathbb{H}^{0,1}(X, T^{1,0}X)$ with respect to the Calabi-Yau metric over $X$.

\subsection{Canonical Family of $(n,0)$-Classes}\label{canonicalfamily}
Based on the construction of the smooth family $\Phi(t)$ of Beltrami differentials in Theorem \ref{BTT}, we
can construct a canonical family of holomorphic
$(n,0)$-forms on the deformation spaces of Calabi-Yau manifolds. Here we just list the
results we need, the reader can refer \cite[Section 5.1]{LRY} for details.

\noindent Let $X$ be an $n$-dimensional Calabi-Yau
manifold and $\{\varphi_{1}, \cdots,
\varphi_{N}\}\in\H^{0,1}(X,T^{1,0}X)$ a basis  where  $N=\dim
\H^{0,1}(X,T^{1,0}X)$.  As constructed in Theorem \ref{BTT}, there
exists a smooth family of Beltrami differentials in the following
form
\[
\Phi(t)=\sum_{i=1}^{N}\varphi_{i}t_{i}+\sum_{|I|\geq2}\varphi_{I}t^{I}
=\sum_{\nu_{1}+\cdots+\nu_{N}\geq1}\varphi_{\nu_{1}\cdots\nu_{N}}t^{\nu_{1}%
}_{1}\cdots t^{\nu_{N}}_{N}\in A^{0,1}(X,T^{1,0}_X)
\]
for $t\in \mathbb C^N$ with $|t|< \epsilon$.
 It is easy to check that the map
\beq e^{\Phi(t)}\lrcorner:\, A^{0}(X,K_{X})\rightarrow
A^{0}(X_{t},K_{X_{t}})\label{iso}\eeq is a well-defined linear
isomorphism.

\begin{proposition} \label{holcriteria}For any smooth $(n,0)$-form
$\Omega\in A^{n,0}(X)$, the section $e^{\Phi(t)}\lrcorner\Omega\in
A^{n,0}(X_{t})$ is holomorphic with respect to the complex structure
$J_{\Phi(t)}$ induced by $\Phi(t)$ on $X_{t}$ if and only if
\begin{equation}
\label{m=1}\overline{\partial}\Omega+\partial(\Phi(t)\lrcorner\Omega)=0.
\end{equation}
\begin{proof}
This is a direct consequence of the following formula, which is \cite[Corollary 3.5]{LRY},
$$e^{- \Phi(t)} \lrcorner d ~( e^{\Phi(t)} \lrcorner \Omega)
=\bp\Omega+\p(\Phi(t)\lrcorner
\Omega)\label{rec1}.$$
\noindent In fact, the operator $d$ can be
decomposed as $d=\overline{\partial}_t+{\partial}_t$, where
$\overline{\partial}_t$ and ${\partial}_t$ denote the $(0,1)$-part
and $(1,0)$-part of $d$, with respect to the complex structure
$J_{\Phi(t)}$ induced by $\Phi(t)$ on $X_{t}$. Note that
$e^{\Phi(t)}\lrcorner\Omega\in A^{n,0}(X_{t})$ and so
$$ \p_t (e^{\Phi(t)}\lrcorner
\Om)=0.$$ Hence,
$$e^{- \Phi(t)} \lrcorner \bar{\partial}_t ~( e^{\Phi(t)} \lrcorner \Omega)
=\bp\Omega+\p(\Phi(t)\lrcorner
\Omega)\label{rec1},$$ which implies the assertion.
\end{proof}
\end{proposition}

\begin{theorem}\label{canonicalform}
Let $\Omega$ be a nontrivial holomorphic $(n,0)$-form
on the Calabi-Yau manifold $X$ and  $X_{t}=(X_{t},J_{\Phi(t)})$ be the
deformation of $X$ induced by the
smooth family $\Phi(t)$ of Beltrami differentials on $X$ as
constructed in Theorem \ref{BTT}. Then, for
$|t|< \epsilon $, \beq \Omega^{c}(t):=e^{\Phi(t)}\lrcorner\Omega\eeq defines a
canonical family of holomorphic $(n,0)$-forms on $X_{t}$
which depends on $t$ holomorphically.
\end{theorem}
\begin{proof}
Since $\Omega$ is holomorphic, and $\Phi(t)$ is
smooth, by Proposition \ref{holcriteria}, we only need to show that
$$\partial (\Phi(t) \lrcorner \Omega)=0$$ in the distribution sense.
In fact, for any test form $\eta$ on $X$,
$$(\Phi(t) \lrcorner \Omega, \partial^* \eta)=
\lim_{k \rightarrow \infty}\left(\left(\sum_{|I| \leq k}\phi_I t^I
\right)\lrcorner \Omega, \partial^* \eta\right)=
\lim_{k \rightarrow \infty} \left( \sum_{i=1}^N t_i \phi_i \lrcorner \Omega
+ \sum_{2 \leq |I|\leq k} t^I \partial \psi_I , \eta\right)=0,$$
as $\phi_i \lrcorner \Omega, ~1 \leq i \leq N$ are harmonic and
$\phi_I \lrcorner \Omega= \partial \psi_I$ are $\partial$-exact for $|I|\geq 2$
by Theorem \ref{BTT}.
\end{proof}

\bcorollary \label{canonicalclass}
Let $\Omega^{c}(t):=e^{\Phi(t)}\lrcorner\Omega$ be the
canonical family of holomorphic $(n,0)$-forms as
constructed in Theorem \ref{canonicalform}.  Then for $|t|<\epsilon$, there holds
the following expansion of $[\Omega^c(t)]$ in cohomology
classes,
\begin{eqnarray} \label{calclass}
[\Omega^c(t)]=[\Omega]+\sum_{i=1}^{N}[\varphi_{i}\lrcorner\Omega]t_{i}+O(|t|^{2}),
\end{eqnarray}
where $O(|t|^{2})$ denotes the terms in
$\displaystyle\bigoplus_{j=2}^{n}H^{n-j,j}(X)$ of order at least
$2$ in $t$.
\ecorollary

\section{Hermitian Symmetric Space and Quantum Correction}\label{symmetric}
In Section \ref{symmetric space}, we review the definitions of locally Hermitian
symmetric spaces and globally Hermitian symmetric spaces. In Section \ref{quantum correction},
we define the quantum correction over the Teichm\"uller space of polarized and marked
Calabi-Yau manifolds, which originally comes from the quantum correction of Yukawa coupling
in the Kodaira-Spencer theory developed in \cite{BCOV}. The definition of quantum correction
also applies to polarized Hyperk\"ahler manifolds.

\subsection{Hermitian Symmetric Space}\label{symmetric space}
 First let us review some basic definitions of symmetric spaces, the reader can refer to
\cite[Chapter 11]{KN} or \cite[Chapter 3]{Z} for details.
Let $N$ be a Riemannian manifold, $p \in N$,
and $r_p > 0$ the injective radius at point $p$. Consider the diffeomorphism $s_p$ from the geodesic ball
$B_{r_p}(p)$ onto $N$ defined by
\begin{equation}
s_p(\exp_p (X))=\exp_p(-X), ~~\forall X\in B_{r_p}(0) \subset T_pN.
\end{equation}

\noindent The map $s_p$ is called the geodesic symmetry at $p$. It has $p$ as an isolated fixed point, and
$(s_p)_{*p}=-id$. In general, it is not an isometry.

\bdefinition
A Riemannian manifold $N$ is called a locally Riemannian symmetric space, if
for any point $p \in N$, the geodesic symmetry $s_p$ is an isometry on $B_{r_p}(p)$.
$N$ is called a globally Riemannian symmetric space if, for any point $p \in N$, there
exists an isometry in its isometry group $I(N)$ whose restriction on $B_{r_p}(p)$ is $s_p$.
\edefinition

\noindent Clearly, globally Riemannian symmetric spaces are locally Riemannian symmetric spaces.
Applying the theorem of Cartan-Ambrose-Hicks \cite{cartan, Ambrose56, Hicks59, Hicks66}
to the map $s_p$ and the isometry $I =-id$ at $T_p N$, we immediately get the following lemma

\blemma\label{mainlemma}
A Riemannian manifold $N$ is a locally Riemannian symmetric space if and only if $\bigtriangledown R=0$,
 i.e., the curvature tensor is parallel. Also, if a locally Riemannian symmetric space is complete and simply-connected,
then it is a globally Riemannian symmetric space. Two locally Riemannian symmetric
spaces are locally isometric if they have the same curvature at one point.
\elemma

\noindent Now, let us consider the complex case,

\bdefinition
A Hermitian manifold $N$ is a locally Hermitian symmetric space if, for any point $p \in N$, $s_p:\exp_p(X)\rightarrow
\exp_p(-X), \forall X \in T_p N$ is a local automorphism around $p$ of $N$, i.e., $s_p$ leaves
its Levi-Civita connection $\nabla$ and complex structure $J$ invariant. It is called a globally Hermitian
symmetric space if it is connected and for any point $p \in N$ there exists an involutive automorphism
$s_p$ of $N$ with $p$ as an isolated fixed point.
\edefinition

\noindent Similarly, in terms of the curvature tensor,~ we have the following characterization
of locally Hermitian symmetric spaces.
\begin{thm}
A Hermitian manifold is a locally Hermitian symmetric space if and only if
\begin{equation}
\nabla R=0=\nabla J.
\end{equation}
where $\nabla$ is the Levi-Civita connection associated to the underlying
Riemannian metric.
\end{thm}

\bcorollary\label{riemanncomplex}
Let $N$ be a K\"ahler manifold , if $N$ is
a locally Riemannian symmetric space, then $N$ is a locally Hermitian
symmetric space.
\ecorollary

\noindent For the Riemannican curvature tensor, Nomizu and Ozeki \cite{NO62} and later Nomizu, without assuming
completeness, proved the following proposition.
\bproposition \label{cur}
(Nomizu and Ozeki \cite{NO62}, Nomizu) For a Riemannian manifold $(N,g)$, if
$\nabla^k R=0$ for some $k \geq 1$, then $\nabla R =0$.
\eproposition

\subsection{Quantum Correction}\label{quantum correction}
In this section, we will define the quantum correction over the Teichm\"uller space
of polarized and marked Calabi-Yau manifolds.
Our motivation for quantum correction comes from the Kodaira-Spencer theory developed
in \cite[Chapter 5]{BCOV}.

\noindent The physical fields of the Kodaira-Spencer theory are differential forms of type $(0,1)$ on $X$
 with coefficients $(1,0)$-vectors, i.e.,
sections $\psi \in C^{\infty}(X, T^{* 0,1} X \otimes T^{1,0} X)$ with
\begin{eqnarray*}
\partial (\psi \lrcorner \Omega)=0,
\end{eqnarray*}
where $\Omega$ is a nowhere vanishing $(n,0)$-form normalized as in Lemma \ref{isometry}.
 Then the Kodaira-Spencer action is given as follows
\begin{eqnarray*}
\lambda^2 S(\psi, \phi| p)= \frac{1}{2} \int_X \psi\lrcorner \Omega
\wedge \frac{1}{\partial} \bar{\partial} (\psi \lrcorner \Omega)
+ \frac{1}{6} \int_X ((\psi+ \phi)\wedge (\psi+\phi))\lrcorner \Omega
\wedge (\psi+\phi)\lrcorner \Omega,
\end{eqnarray*}
where $\lambda$ is the coupling constant. The Euler-Lagrange equation of this action is
\begin{eqnarray*}
\bp (\phi \lrcorner \Omega) + \frac{1}{2}
\p ((\psi+ \phi)\wedge (\psi+\phi))\lrcorner \Omega=0.
\end{eqnarray*}

\noindent They also discussed that the Kodaira-Spencer action is a closed string theory action
at least up to cubic order. In a properly regularized Kodaira-Spencer theory,
the partition function should satisfy
\begin{eqnarray}
e^{W(\lambda, \phi| t, \bar{t})}= \int D\psi e^{S(\lambda, \phi| t, \bar{t})}.
\end{eqnarray}
The effective action $W(\lambda, \phi| t, \bar{t})$ is physically computed in \cite{BCOV}
in the flat affine coordinate $t=(t_1, t_2, \cdots, t_N)$. The term $W_0(\lambda, \phi| t, \bar{t})$
in front of $\lambda^{-2}$ satisfies
\begin{eqnarray}
W_0(\lambda, \phi| t, \bar{t})= \lambda^2 S_0(\phi, \psi| t, \bar{t}),
\end{eqnarray}
where $\psi(t)$ and $W_0(\lambda, \phi| t, \bar{t})$ satisfy
\begin{eqnarray}
\frac{\partial \psi(t)}{\partial t_i}|_{t=0}=
\phi_i, ~~~\bp(\psi \lrcorner \Omega) + \frac{1}{2} ((\phi+ \psi(t))\wedge (\phi +\psi(t)))\lrcorner \Omega=0
\end{eqnarray}
and
\begin{eqnarray}\label{12}
\frac{\partial^3 W_0(\phi| t, \bar{t})}{\partial t_i \partial t_j \partial t_k}
=C_{ijk}(t_1, t_2, \cdots, t_N).
\end{eqnarray}

\noindent $W_0(\phi|t, \bar{t})$ may be viewed as the effective action for the massless modes $\phi$
from which the massive modes will be integrated out. It is quite amazing that integrating the massive modes has only
the effectivity of taking derivatives of the Yukawa coupling. For example the four point function give rise to
$\nabla_l C_{ijk}$, the five point function to $\nabla_s \nabla_l C_{ijk}$ and the six point
function to $\nabla_r \nabla_s \nabla_l C_{ijk}$. Thus all the discussion suggests us to define
the quantum correction of the Yukawa coupling as
\begin{eqnarray}\label{physical}
\sum_{s_1 + \cdots+ s_N=1}^{\infty} C_{ijk, s_1, \cdots, s_N} t_1^{s_1}\cdots t_N^{s_N}.
\end{eqnarray}

\noindent  On the other hand, besides the canonical family of holomorphic $(n,0)$-forms, we can define
the classic canonical family as
\begin{eqnarray}
\Omega^{cc}(t) = \exp(\sum_{i=1}^N t_i \phi_i) \lrcorner \Omega \in A^n(X, \mathbb{C}).
\end{eqnarray}

\begin{proposition} \label{physicalqc}
Fix $p\in \mathcal{T}$, let $X$ be the corresponding fiber in the versal
family $ \mathcal{X} \rightarrow \mathcal{T}$.
Let $\Omega$ be a nontrivial holomorphic $(n,0)$-form over $X$
and $\{\phi_i\}_{i=1}^N$ be an orthonormal basis of $\mathbb{H}^{0,1}(X, T^{1,0}X)$
with respect to the Calabi-Yau metric.
If the cohomology class
\begin{eqnarray*}
[\Xi(t)]=[\Omega^c(t)]-[\Omega^{cc}(t)]=[\Omega^c(t)]-[\exp(\sum_{i=1}^N t_i \phi_i) \lrcorner \Omega]=0,
\end{eqnarray*}
then the quantum correction of the Yukawa coupling vanishes, i.e.,
\begin{eqnarray*}
\sum_{s_1 + \cdots+ s_N=1}^{\infty} C_{ijk, s_1, \cdots, s_N} t_1^{s_1}\cdots t_N^{s_N}=0.
\end{eqnarray*}
 Moreover, $\suml_{i,j,k=1}^N t_i t_j t_k[\phi_i \lrcorner \phi_{jk} \lrcorner \Omega]=0$
if and only if the first order quantum correction of the Yukawa coupling
\begin{eqnarray*}
\sum_{s_1 + \cdots+ s_N=1} C_{ijk, s_1, \cdots, s_N} t_1^{s_1}\cdots t_N^{s_N}=0.
\end{eqnarray*}
\bproof
From the definition of Yukawa coupling, with the flat affine coordinate $t=(t_1, \cdots, t_N)$, we have
\begin{eqnarray*}
C_{ijk}(t_1, t_2, \cdots , t_N) &=&
\int_X \Omega^c(t) \wedge \frac{\partial^3 \Omega^c(t)}
{\partial t_i \partial t_j \partial t_k}\\
&=& \int_X \Omega^c (t) \wedge \frac{\partial^3 }
{\partial t_i \partial t_j \partial t_k} \left(\exp(\sum_{i=1}^N t_i \phi_i)\lrcorner \Omega+ \Xi(t) \right),
\end{eqnarray*}
where $\Xi(t)=\Omega^c(t)-\exp(\suml_{i=1}^N t_i \phi_i) \lrcorner \Omega$.
Direct computation shows that $$\frac{\partial^3 }
{\partial t_i \partial t_j \partial t_k} \left(\exp(\suml_{i=1}^N t_i \phi_i)\lrcorner\Omega \right)= \phi_i \lrcorner
\phi_j \lrcorner \phi_k \lrcorner \Omega,$$
thus the term
$$\int_X \Omega^c(t) \wedge \frac{\partial^3 }
{\partial t_i \partial t_j \partial t_k} \left(\exp(\suml_{i=1}^N t_i \phi_i) \lrcorner \Omega\right)$$
has order zero with respect to $t$.
Thus the quantum correction of Yukawa coupling satisfies
\begin{eqnarray*}
\sum_{s_1 + \cdots+ s_N=1}^{\infty} C_{ijk, s_1, \cdots, s_N} t_1^{s_1}\cdots t_N^{s_N}
= \int_X \Omega^c (t) \wedge \Xi(t).
\end{eqnarray*}
Therefore, $[\Xi(t)]=0$ implies that the quantum correction of the Yukawa coupling vanishes.
\noindent Moreover, we have
\begin{eqnarray*}
&&\suml_{s_1 + \cdots+ s_N=1} C_{ijk, s_1, \cdots, s_N} t_1^{s_1}\cdots t_N^{s_N}=0;\\
& \Longleftrightarrow &
\int_X \Omega \wedge \phi_i \lrcorner \phi_j \lrcorner \phi_{I}\lrcorner \Omega
+ \int_X \phi_i \lrcorner \Omega \wedge \phi_j \lrcorner \phi_I \lrcorner \Omega=0
~~~\text{for any}~~~ 1 \leq i, j \leq N  ~~~\text{and}~~~  |I|=2;\\
& \Longleftrightarrow &
\int_X \phi_i \lrcorner \Omega \wedge \phi_j \lrcorner \phi_I \lrcorner \Omega=0,
~~~\text{for any}~~~ 1 \leq i, j \leq N  ~~~\text{and}~~~  |I|=2; \\
&&\left(\text{as}~ \phi_i \lrcorner \Omega \wedge \phi_j \lrcorner \phi_I \lrcorner \Omega
= \Omega\wedge \phi_i \lrcorner \phi_j \lrcorner \phi_I \lrcorner \Omega \right);\\
&\Longleftrightarrow &
\mathbb{H}(\phi_j \lrcorner \phi_I \lrcorner \Omega)=0,
~~~\text{for any}~~~ 1 \leq j \leq N ~~~\text{and}~~~ |I|=2; \\
&& \left(\text{as} ~~~ \{[\phi_i \lrcorner \Omega]\}_{i=1}^N
~~~\text {is a basis of}~~~ \mathbb{H}^{2,1}(X)\right);\\
&\Longleftrightarrow &
\suml_{i,j,k=1}^N t_i t_j t_k[\phi_i \lrcorner \phi_{jk} \lrcorner \Omega]=0.
\end{eqnarray*}
\eproof
\end{proposition}

\begin{lemma}
Under the conditions as Proposition \ref{physicalqc}, the form $\Xi(t)$ are identically zero, i.e., $\Xi(t)=0$ if and only if $[\phi_i, \phi_j]=0$
for all $1 \leq i, j \leq N$. And, for $|t|< \epsilon$, there holds the following expansion of $[\Xi(t)]$
in cohomology classes,
\begin{eqnarray}
[\Xi(t)]= \sum_{i,j,k=1}^N t_i t_j t_k [\phi_i \lrcorner \phi_{jk} \lrcorner \Omega] + O(|t|^4).
\end{eqnarray}
where $O(|t|^4)$ denotes the terms of order at least $4$ in $t$.
\bproof
From the construction of the smooth family \ref{betrami} of Beltrami differentials, see \cite[page 162] {Morrow-Kodaira71} or \cite{To}, we have
\begin{eqnarray}  \label{induction}
\phi_{K}= -\frac{1}{2} \bar{\partial}^* G  (\sum_{I+J=K} [\phi_I, \phi_J] )
~~~~\text{for}~~ |K| \geq 2.
\end{eqnarray}
 Thus we have
\begin{eqnarray*}
\Xi(t)=0  & \Longleftrightarrow & \left(\exp\left(\sum_{|I| \geq 1} \phi_I t^I\right)
-\exp\left(\sum_{i=1}^N \phi_i t_i\right) \right)\lrcorner \Omega=0\\
&\Longleftrightarrow & \sum_{|I| \geq 1} \phi_I t^I= \sum_{i=1}^N \phi_i t_i \\
& \Longleftrightarrow &  \phi_I =0   ~~~~~\text{for}~~~ |I|\geq 2 \\
& \Longleftrightarrow & [\phi_i, \phi_j]=0~~~~\text{for} ~~~ 1 \leq i, j \leq N ~~~~\text{by Formula} \, (\ref{induction}).
\end{eqnarray*}

\noindent  Moreover, by the property that $\phi_I \lrcorner \Omega = \partial \psi_I $ for $|I| \geq 2$ by Theorem \ref{BTT}, the cohomology class
of the quantum correction satisfies
\begin{eqnarray*}
[\Xi(t)] &=& \left[\left(\exp\left(\sum_{|I| \geq 1} \phi_I t^I\right)
-\exp\left(\sum_{i=1}^N \phi_i t_i\right)\right)\lrcorner \Omega \right] \\
&=& \left [\sum_{i, j=1}^N (\phi_{ij} \lrcorner \Omega) t_i t_j + \sum_{i, j, k
=1}^N (\phi_i \lrcorner \phi_{jk} \lrcorner \Omega + \phi_{ijk} \lrcorner \Omega) t_i t_j t_k  + O(|t|^4) \right]\\
&=& \sum_{i,j,k=1}^N t_i t_j t_k [\phi_i \lrcorner \phi_{jk} \lrcorner \Omega] + O(|t|^4).
\end{eqnarray*}
\eproof
\end{lemma}

\noindent  Thus the lowest order quantum correction has the form
$\sum_{i,j,k=1}^N t_i t_j t_k [\phi_i \lrcorner \phi_{jk} \lrcorner \Omega]$,
so we have the following definition,
\noindent \bdefinition We define the cohomology class
$[\Xi(t)]$
to be the strong quantum correction at point $p \in \mathcal{T}$ and the cohomology class
$$[\Xi(t)]_1=\sum_{i,j,k=1}^N t_i t_j t_k [\phi_i \lrcorner \phi_{jk} \lrcorner \Omega],$$
to be the weak quantum correction at point $p \in \mathcal{T}$.
If $[\Xi(t)]_1=0,$
we will say that there is no weak quantum correction
at point $p \in \mathcal{T}$.
Moreover, if $[\Xi(t)]=0,$
we will say that there is no strong quantum correction at point $p \in \mathcal{T}$.
\edefinition

\begin{remark}
For Calabi-Yau threefolds, by Proposition \ref{physicalqc}, no strong quantum correction at point $p \in \mathcal{T}$
implies that there is no quantum correction of Yukawa coupling at point $p \in \mathcal{T}$.
Moreover, no weak quantum correction at point $p \in \mathcal{T}$ is equivalent to that
the first order quantum correction of Yukawa coupling at $p \in \mathcal{T}$ vanishes.
\end{remark}

\section{Quantum Correction and the Weil-Petersson Metric}\label{qc and wpmetric}
In Section \ref{wpmetric} and Section \ref{cuvature}, we review the
Weil-Petersson metric over the Teichm\"uller space $\mathcal{T}$
and derive a local formula for $\nabla_r R$ and $\nabla_{\overline{r}} R$ in
the flat affine coordinate. In Section \ref{wpsymmetric},
under the assumption of no weak quantum correction
 at any point $p \in \mathcal{T}$, we prove
that $\mathcal{T}$ is a locally Hermitian symmetric space with the Weil-Petersson metric
by using the formula of $\nabla_r R$ and $\nabla_{\overline{r}} R$.
The results in this section also hold for polarized Hyperk\"ahler manifolds.

\subsection{The Weil-Petersson Geometry}\label{wpmetric}
The local Kuranishi family of polarized Calabi-Yau manifolds $\pi: \mathcal{X}\rightarrow S$
is unobstructed by the Bogomolov-Tian-Todorov theorem \cite{Ti, To}. One can assign the
unique Ricci-flat or Calabi-Yau metric $g(s)$ on the fiber $X:=X_s$ in the polarization
K\"ahler class \cite{Y}. Then, on the fiber $X$, the Kodaira-Spencer theory
gives rise to an injective map $\rho: T_s S \longrightarrow H^1(X,T^{1,0}X)\cong
\mathbb{H}^{0,1}(X, T^{1,0}X)$, the space of harmonic representatives.
The metric $g(s)$ induces a metric on $A^{0,1}(X,T^{1,0}X)$. The reader may also refer to \cite{W} for the discussion. For $v,w \in T_s S$,
one then defines the Weil-Petersson metric on S by
\begin{equation}
g_{WP}(v,w):=\int_{X} \langle \rho(v), \rho(w)\rangle_{g(s)} dvol_g(s).
\end{equation}

\noindent Let $\dim X=n$,  by the fact that the global holomorphic $(n,0)$-form $\Omega:=\Omega(s)$ is
flat with respect to $g(s)$, it can be shown \cite{Ti} that
\begin{equation}\label{WP}
g_{WP}(v,w)=-\frac{\tilde{Q}(i(v)\Omega, \overline{i(w)\Omega})}{\tilde{Q}(\Omega, \overline{\Omega})}.
\end{equation}

\noindent Here, for convenience, we write $\tilde{Q}(\cdot,\cdot)=(\sqrt{-1})^n Q(\cdot, \cdot)$, where
$Q$ is the intersection product. Therefore, $\tilde{Q}$ has alternating signs in the successive
primitive cohomology groups $H_{pr}^{p,q}\subset H^{p,q},~~ p+q=n$, which we simply denote by $H^{p,q}$ for convenience.

\noindent The formula $(\ref{WP})$ implies that the natural map $H^1(X,T^{1,0}X) \longrightarrow \mbox{Hom} (H^{n,0}, H^{n-1,1})$
via the interior product $v \mapsto v \lrcorner \Omega$ is an isometry from the tangent space
$T_s S$ to $(H^{n,0})^* \otimes H^{n-1,1}$. So the Weil-Petersson metric is precisely the metric induced
from the first piece of the Hodge metric on the horizontal tangent bundle over the
period domain. Let $F^n$ denote the Hodge bundle induced by $H^{n,0}$. A simple calculation in formal Hodge theory shows that
\begin{equation}
\omega_{WP}=Ric_{\tilde{Q}}(F^n)=-\partial \bar{\partial} \log \tilde{Q}(\Omega, \overline{\Omega})
=- \frac{\tilde{Q}(\p_i \Omega, \overline{\p_j \Omega})}{\tilde{Q}(\Omega, \overline{\Omega})}
+ \frac{\tilde{Q}(\p_i \Omega, \overline{\Omega})~ \tilde{Q}(\Omega, \overline{\p_j \Omega})}
{\tilde{Q}(\Omega, \overline{\Omega})^2},
\end{equation}
where $\omega_{WP}$ is the $2$-form associated to $g_{WP}$. In particular, $g_{WP}$ is K\"ahler and is
independent of the choice of $\Omega$. In fact, $g_{WP}$ is also independent of the choice of the polarization.
Next, we define
\be
K_i =- \p_i \log \tilde{Q}(\Omega, \overline{\Omega})=-\frac{\p_i \Omega, \overline{\Omega}}
{\tilde{Q}(\Omega, \overline{\Omega})}
\ee
and
\be
D_i \Omega= \p_i \Omega +K_i \Omega
\ee
for $1 \leq i \leq N$. And it is easy to check that $D_i \Omega$ is the projection of $\p_i \Omega$ into
$H^{n-1,1}$ with respect to the quadratic form $\tilde{Q}(\cdot, \cdot)$.
And if we denote the Christoffel
symbol of the Weil-Petersson metric by $\Gamma^k_{ij}$, it is easy to check that
\be
D_j D_i \Omega = \p_j D_i \Omega - \Gamma^k_{ij} D_k \Omega + K_jD_i \Omega,
\ee
is the projection of $\p_j D_j \Omega$ into $H^{n-2,2}$.

\subsection{Property of the Curvature Tensor}\label{cuvature}

\noindent To simplify the notation, we abstract the discussion by considering a variations of polarized Hodge structure
$H \longrightarrow S$ of weight n with $h^{n,0}=1$ and a smooth base $S$. Also, we always assume that
it is effectively parametrized in the sense that the infinitesimal period map
\begin{equation}
\Phi_{*, s}: T_s S \longrightarrow \mbox{Hom} (H^{n,0}, H^{n-1,1})\oplus \mbox{Hom}(H^{n-1,1}, H^{n-2,2}) \oplus \cdots
\end{equation}
be injective in the first piece. Then the Weil-Petersson metric $g_{WP}$ on S is defined by formula (\ref{WP}).
In our abstract setting, instead of using $H_{pr}^{p,q}$ in the geometric case, we will write $H^{p,q}$ directly for simplicity.

\begin{thm}\label{LSS}
For a given effectively parametrized ploarized variation of
Hodge structure $H\rightarrow S$
of weight $n$ with $h^{n,0}=1$ and smooth $S$,
the Riemannian curvature of the Weil-Petersson metric $g_{WP}$ on $S$ satisfies:

\begin{enumerate}
\item
Its Riemannian curvature tensor is
\be
R_{i\overline{j}k\overline{l}}=g_{ij}g_{kl}+g_{il}g_{kj}-
\frac{\tilde{Q}(D_k D_i \Omega, \overline{D_l D_j \Omega})}{\tilde{Q}(\Omega, \overline{\Omega})}.
\ee

\item The covariant derivative of the Riemannian curvature tensor is
\be
\nabla_r R_{i \bar{j} k\bar{l}}&=& \frac{\tilde{Q}(\p_r \p_k \p_i \Omega,
 \overline{D_l D_j \Omega})}{\tilde{Q}(\Omega, \overline{\Omega})};\\
\nabla_{\bar{r}}R_{i\bar{j}k\bar{l}}&=& \frac{\tilde{Q}(D_k D_i \Omega,
\overline{\p_r \p_l \p_j \Omega})}{\tilde{Q}(\Omega, \overline{\Omega})}.
\ee
\end{enumerate}
\end{thm}

\noindent The main idea of the proof is that, when we use the canonical family of $(n,0)$-classes constructed in Corollary
\ref{canonicalclass} to express the Weil-Petersson metric, the flat affine coordinate
 $t= (t_1, \cdots, t_N )$ is normal at the point $t=0$.
Since the problem is local, we may assume that $S$
is a disk in $\mathbb{C}^N$, where $N=\mbox{dim}~ H^{n-1,1}$, around $t=0$. The first part of the above theorem for the curvature formula of the Weil-Petersson metric is due to Strominger. See \cite{W}.

\begin{proof}
Let $\Omega^c(t)$ be the canonical family of holomorphic $(n,0)$-forms
 constructed in Theorem \ref{canonicalform}, so we have
\be
\Omega^c(t) &=&\Omega+ \sum^N_{i=1} t_i \varphi_i \lrcorner \Omega +
\sum_{\mid I\mid \geq 2} t^I ~\varphi_I \lrcorner \Omega
+ \sum_{k\geq 2} \wedge^k\Phi(t)\lrcorner \Omega\\
&=&  \Omega +\sum_{i=1}^N t_i \varphi_i \lrcorner\Omega + \frac{1}{2!} \sum_{i,j=1}^N t_i t_j
(\varphi_i \lrcorner\varphi_j \lrcorner \Omega+ \varphi_{ij} \lrcorner \Omega)\\
&&+ \frac{1}{3!} \sum_{i, j< k}t_i t_j t_k (\varphi_i \lrcorner \varphi_j \lrcorner \varphi_k \lrcorner \Omega
+\varphi_i \lrcorner \varphi_{jk} \lrcorner \Omega + \varphi_{ijk} \lrcorner \Omega)+ O(\mid t \mid^4)\\
&=&a_0 +\sum_{i=1}^N a_i t_i + \cdots +\sum_{\mid I \mid =k}a_I t^I + \cdots.
\ee
And the coefficients satisfy $\tilde{Q}(a_0, a_0)=1,  \tilde{Q}(a_i, a_j)=-\delta_{ij}$ and
$\tilde{Q}(a_0, a_i)= \tilde{Q}(a_0, a_I)= \tilde{Q}(a_i, a_I)=0$ for $|I| \geq 2$.
For multi-indices $I$ and $J$,
we set $q_{I,\bar{J}}:=\tilde{Q}(a_I, \bar{a}_J)$. Then we have
\begin{eqnarray*}
q(t):&=&\tilde{Q}(\Omega^c(t), \bar{\Omega^c(t)})\\
&=& 1-\sum_i t_i \bar{t_i} + \sum_{i,j,k,l}
\frac{1}{2!2!}q_{ik,\bar{jl}}t_i t_k \bar{t_j} \bar{t_l} \\&&
+ \sum_{i,j,k,l,r}\frac{1}{2!3!}q_{ik,\bar{jlr}}t_i t_k \bar{t_j t_l t_r}
+\sum_{i,j,k,l,r} \frac{1}{2!3!}q_{ikr, \bar{jl}}t_i t_k t_r \bar{t_j t_k}
+ O(t^6),
\end{eqnarray*}
where
\be
q_{ik, \bar{jl}}&=&\tilde{Q}(\varphi_i \lrcorner \varphi_k \lrcorner \Omega,~~
\varphi_j \lrcorner \varphi_l \lrcorner \Omega),\\
q_{ik, \overline{jlr}}&=& \frac{1}{3}\tilde{Q}(\varphi_i \lrcorner \varphi_k \lrcorner \Omega,~~
\varphi_j \lrcorner \varphi_{lr} \lrcorner \Omega
+ \varphi_l \lrcorner \varphi_{jr}\lrcorner\Omega
+\varphi_r \lrcorner \varphi_{jl} \lrcorner \Omega),\\
q_{ikr, \overline{jl}}&=& \frac{1}{3} \tilde{Q}(\varphi_i \lrcorner \varphi_{kr}\lrcorner \Omega
+ \varphi_k \lrcorner \varphi_{ir} \lrcorner \Omega
+\varphi_r \lrcorner \varphi_{ik}\lrcorner \Omega,~~
\varphi_j \lrcorner \varphi_l \lrcorner \Omega).
\ee
Thus, the Weil-Petersson metric can be expressed as
\begin{eqnarray*}
g_{k\bar{l}}&=&-\partial_k \partial_{\bar{l}}~\log ~q
=q^{-2}(\partial_k q \partial_{\bar{l}}q- q\partial_k \partial_{\bar{l}}q)\\
&=&(1+2\sum_i t_i \bar{t_i}+\cdots)[t_l t_{\bar{k}}-(1-\sum_i t_i \bar{t_i})
(-\delta_{kl}+\sum_{i,j}q_{ik,jl}t_i \bar{t_j}\\&&
+ \sum_{i,j,r} q_{ik, \overline{jlr}}t_i \overline{t_j t_r}
+ \sum_{i,j,r}  q_{ikr, \overline{jl}}t_i t_r \overline{t_l}+\cdots)]\\
&=&\delta_{kl}+\delta_{kl}\sum_i t_i \bar{t_i} + t_l\bar{t_k}
-\sum_{i,j}q_{ik,jl}t_i \bar{t_j}
+\sum_{i,j,r}q_{ik, \overline{jlr}}t_i \overline{t_j t_r}
+ \sum_{i,j,r} q_{ikr, \overline{jl}}t_i t_r \overline{t_l}+\cdots .
\end{eqnarray*}

\noindent As a result, the Weil-Petersson
metric $g$ is already in its geodesic normal form at $t=0$, so
the Christoffel symbols at point $t=0$ is zero, i.e. $\Gamma_{ij}^k(0)=0$.
So the full curvature tensor at $t=0$ is given by
\be
R_{i\overline{j}k\overline{l}}(0)=\frac{\p^2 g_{k\overline{l}}}{\p t_i \p \overline{t_j}}(0)
=\delta_{ij} \delta_{kl} + \delta_{il}\delta_{kj} +q_{ik, \overline{jl}}.
\ee
Rewrite this in its tensor form then gives the formula in the theorem.

\noindent By using the well-known formula:
\begin{eqnarray*}
\bigtriangledown_r R_{i\bar{j}k\bar{l}}=\frac{\partial}{\partial t_r} R_{i\bar{j}k\bar{l}}
- \Gamma^q_{ri} R_{q\bar{j}k\bar{l}}-\Gamma^q_{rk} R_{i\bar{j}q \bar{l}},\\
\bigtriangledown_{\bar{r}} R_{i\bar{j}k\bar{l}}=\frac{\partial}{\partial \bar{t_r}} R_{i\bar{j}k\bar{l}}
- \overline{\Gamma^q_{rj}} R_{i \bar{q}k\bar{l}}-\overline{\Gamma^q_{rl}} R_{i\bar{j}k \bar{q}},
\end{eqnarray*}
at the point $t=0$, we have
\begin{eqnarray*}
\bigtriangledown_r R_{i\bar{j}k\bar{l}}(0)=\frac{\partial}{\partial t_r} R_{i\bar{j}k\bar{l}}(0);~~~
\bigtriangledown_{\bar{r}} R_{i\bar{j}k\bar{l}}(0)=\frac{\partial}{\partial \bar{t_r}} R_{i\bar{j}k\bar{l}}(0).
\end{eqnarray*}
as $\Gamma_{ij}^k(0)=0$ for any $1 \leq i,j,k \leq N$. And,  from the formula of Riemannian curvature for K\"ahler manifold
\begin{equation}
R_{i\bar{j}k \bar{l}}=\frac{\partial^2 g_{i\bar{j}}}{\partial t_k \partial \overline{t_l}}- g^{p\bar{q}}
\frac{\partial g_{i \bar{q}}}{\partial t_k} \frac{\partial g_{j \bar{k}}}{\partial \overline{t_l}},
\end{equation}
each term of $\nabla_r R_{i \bar{j}k\bar{l}}$ or $\nabla_{\bar{r}} R_{i \bar{j}k\bar{l}}$
 includes degree 1 term as a factor, except
$\frac{\partial^3 g_{k\bar{l}}}
{\partial t_i \partial \overline{t_j}\partial t_r}$
, $\frac{\partial^3 g_{k\bar{l}}}{\partial t_i \partial \overline{t_j}\partial \overline{t_r}},$
thus it is zero at $t=0$ from the expression of $g_{k\overline{l}}$. So we have
\be
\nabla_r R_{i \bar{j}k\bar{l}}(0)=\frac{\partial^3 g_{k\bar{l}}}
{\partial t_i \partial \overline{t_j}\partial t_r}(0)= q_{ikr, \overline{jl}},\\
\nabla_{\bar{r}} R_{i \bar{j}k\bar{l}}(0)
=\frac{\partial^3 g_{k\bar{l}}}{\partial t_i \partial \overline{t_j}\partial \overline{t_r}}(0)
=q_{ik, \overline{jlr}}.
\ee
Rewrite this in its tensor form, we get the formula.
\end{proof}

\subsection{Quantum Correction and the Weil-Petersson Metric}\label{wpsymmetric}

\noindent In this section, we consider the locally Hermitian symmetric property of the Teichm\"uller
space $\mathcal{T}$ of polarized and marked Calabi-Yau manifolds.

\begin{theorem}
Let $\mathcal{T}$ be the Teichm\"uller space of polarized and marked Calabi-Yau manifolds.
If there is no weak quantum correction
at any point $p \in \mathcal{T}$, i.e., $[\Xi(t)]_1=0$,
 then $\mathcal{T}$
is a locally Hermitian symmetric space with the Weil-Petersson metric.
\end{theorem}
\bproof
Fix $p \in \mathcal{T}$, let $X$ be the corresponding fiber in the versal family $\mathcal{X} \rightarrow \mathcal{T}$.
If there is no weak quantum correction at point $p \in \mathcal{T}$,
i.e., $$[\Xi(t)]_1= \sum_{i+j+k=1} t_i t_j t_k [\phi_i \lrcorner \phi_{jk} \lrcorner \Omega]=0,$$
then $[\phi_i \lrcorner \phi_{jk} \lrcorner \Omega]+ [\phi_j \lrcorner \phi_{ik} \lrcorner \Omega]
+ [\phi_k \lrcorner \phi_{ij} \lrcorner \Omega]=0$ for any $1 \leq i, j,k \leq N$. So, from Theorem \ref{LSS},
we have
\be
\nabla_r R_{i \bar{j}k\bar{l}}(p)=\frac{\partial^3 g_{k\bar{l}}}
{\partial t_i \partial \overline{t_j}\partial t_r}(p)= q_{ikr, \overline{jl}}
=\frac{1}{3} \tilde{Q}(\varphi_i \lrcorner \varphi_{kr}\lrcorner \Omega
+ \varphi_k \lrcorner \varphi_{ir} \lrcorner \Omega
+\varphi_r \lrcorner \varphi_{ik}\lrcorner \Omega,~~
\varphi_j \lrcorner \varphi_l \lrcorner \Omega)=0,\\
\nabla_{\bar{r}} R_{i \bar{j}k\bar{l}}(p)
=\frac{\partial^3 g_{k\bar{l}}}{\partial t_i \partial \overline{t_j}\partial \overline{t_r}}(p)
=q_{ik, \overline{jlr}}= \frac{1}{3}\tilde{Q}(\varphi_i \lrcorner \varphi_k \lrcorner \Omega,~~
\varphi_j \lrcorner \varphi_{lr} \lrcorner \Omega
+ \varphi_l \lrcorner \varphi_{jr}\lrcorner\Omega
+\varphi_r \lrcorner \varphi_{jl} \lrcorner \Omega)=0,
\ee
i.e., $\nabla R =0$. So $\mathcal{T}$ is a locally Hermitian symmetric space by Lemma \ref{mainlemma}.
\eproof

\noindent On the other hand, by the definition of locally Hermitian symmetric spaces, the following
condition can also guarantee the locally Hermitian symmetric property for the Teichm\"uller space $\mathcal{T}$.
\btheorem \label{finite}
Let $\mathcal{T}$ be the Teichm\"uller space of polarized and marked Calabi-Yau manifolds
and $\Omega^c(t)$ the canonical form constructed in Theorem \ref{canonicalform}. If the Weil-Petersson
potential $\tilde{Q}(\Omega^c_t, \overline{\Omega^c_t})$ only has finite terms, i.e., a polynormial in
terms of the flat affine coordinate $t$,  then $\mathcal{T}$
is a locally Hermitian symmetric space with the Weil-Petersson metric.
Furthermore, if $\mathcal{T}$ is complete, then $\mathcal{T}$ is a globally Hermitian
symmetric space with the Weil-Petersson metric.
\etheorem

\bproof
Because of Proposition \ref{cur}, to prove a K\"ahler manifold is a locally Hermitian symmetric space,
we only need to show its curvature tensor satisfies $\nabla^m R=0$ for some positive integer $m$.
If the Weil-Petersson potential $\tilde{Q}(\Omega(t), \overline{\Omega}(t))$
only has finite terms, i.e., it is a polynomial of the flat affine coordinate $t=(t_1,t_2,\cdots, t_N)$.
Then, in the flat affine coordinate $t$,  the coefficients of the Weil-Petersson metric
and its curvature tensor
\be
g_{k\overline{l}}&=&-\p_k \p_{\overline{l}} \log \tilde{Q}(\Omega^c(t), \overline{\Omega^c(t)})\\
R_{i\overline{j}k\overline{l}}&=&\frac{\p^2 g_{i\overline{j}}}{\p t_k \p \overline{t_l}}
-g^{p\overline{q}} \frac{\p g_{i\overline{q}}}{\p t_k}\frac{\p g_{p\overline{j}}}{\p \overline{t_l}}.
\ee
is a polynomial of variable $(t_1,t_2,\cdots,t_N, \overline{t_1}, \overline{t_2}, \cdots, \overline{t_N})$.

\noindent On the other hand, from the proof of Theorem \ref{LSS},
the flat affine coordinate $t$ is a normal coordinate at the point $t=0$.
So we have the Christoffel symbols at the point $t=0$ vanish, i.e., $\Gamma^k_{ij}(0)=0$.
Thus, at the point $t=0$, the covariant derivative $\nabla_p T= \p_p T$ for any $(0,m)$-tensor $T$.
Therefore, for a large enough integer $m$, we have $\nabla^m R(0)=0$. Thus the Teichm\"uller space $\mathcal{T}$
 is a locally Hermitian symmetric space
with the Weil-Petersson metric.
\eproof

\noindent In particular, we have the following corollary,
\begin{corollary}\label{corfinite}
If the canonical family of $(n,0)$-classes $[\Omega^c(t)]$ constructed in Corollary \ref{canonicalclass}
has finite terms, i.e., a polynormial in
terms of the flat affine coordinate $t$,  then $\mathcal{T}$
is a locally Hermitian symmetric space with the Weil-Petersson metric.
\begin{proof}
The proof follows directly from Theorem \ref{finite}.
\end{proof}
\end{corollary}

\section{Quantum Correction and Calabi-Yau Threefolds}\label{qc and hodge}
In Section \ref{period domain}, we review some basic properties of the period domain from Lie group and Lie algebra point of view.
In Section \ref{hodgesymmetric}, for Calabi-Yau threefolds, we show that no strong quantum correction is equivalent to that
the image $\Phi(\mathcal{T})$ of the Teichm\"uller space $\mathcal{T}$ under the period map is an open submanifold of
a globally Hermitian symmetric space with the same dimension as $\mathcal{T}$.
\subsection{Period Domain}\label{period domain}
Let us briefly recall some properties of the period domain from Lie group and Lie algebra point of view. All results in this section are well-known to the experts in the subject. The purpose to give details is to fix notations. One may either skip this section or refer to \cite{GS} and \cite{Sch} for most of the details.

\noindent A pair $(X,L)$ consisting of a Calabi--Yau manifold $X$ of complex dimension $n$ with $n\geq 3$ and an ample
line bundle $L$ over $X$ is called a {polarized Calabi--Yau
manifold}. By abuse of notation, the Chern class of $L$ will also be
denoted by $L$ and thus $L\in H^2(X,\mathbb{Z})$.
The Poincar\'e bilinear form $Q$ on $H_{pr}^n(X,{\mathbb{Q}})$ is
defined by
\begin{equation*}
Q(u,v)=(-1)^{\frac{n(n-1)}{2}}\int_X u\wedge v
\end{equation*}
for any $d$-closed $n$-forms $u,v$ on $X$.
Furthermore, $Q $ is nondegenerate and can be extended to
$H_{pr}^n(X,{\mathbb{C}})$ bilinearly.
Let $f^k=\sum_{i=k}^nh^{i,n-i}$ and
$F^k=F^k(X)=H_{pr}^{n,0}(X)\oplus\cdots\oplus H_{pr}^{k,n-k}(X)$,
from which we have the decreasing filtration
$H_{pr}^n(X,{\mathbb{C}})=F^0\supset\cdots\supset F^n.$ We know that
\begin{eqnarray}
&\dim_{\mathbb{C}} F^k=f^k,  \label{cl45}\\
H^n_{pr}(X,{\mathbb{C}})&=F^{k}\oplus \bar{F^{n-k+1}}, \quad\text{and}\quad
H_{pr}^{k,n-k}(X)=F^k\cap\bar{F^{n-k}}.
\end{eqnarray}
In terms of the Hodge filtration, then the Hodge-Riemann relations are
\begin{eqnarray}
Q\left ( F^k,F^{n-k+1}\right )=0, \quad\text{and}\quad\label{cl50}\\
Q\left ( Cv,\bar v\right )>0 \quad\text{if}\quad v\ne 0,\label{cl60}
\end{eqnarray}
where $C$ is the Weil operator given by $Cv=\left (\sqrt{-1}\right
)^{2k-n}v$ for $v\in H_{pr}^{k,n-k}(X)$. The period domain $D$
for polarized Hodge structures with data \eqref{cl45} is the space
of all such Hodge filtrations
\begin{equation*}
D=\left \{ F^n\subset\cdots\subset F^0=H_{pr}^n(X,{\mathbb{C}})\mid %
\eqref{cl45}, \eqref{cl50} \text{ and } \eqref{cl60} \text{ hold}
\right \}.
\end{equation*}
The compact dual $\check D$ of $D$ is
\begin{equation*}
\check D=\left \{ F^n\subset\cdots\subset F^0=H_{pr}^n(X,{\mathbb{C}})\mid %
\eqref{cl45} \text{ and } \eqref{cl50} \text{ hold} \right \}.
\end{equation*}
The period domain $D\subseteq \check D$ is an open subset.
\noindent Let us introduce the notion of an adapted basis for the given Hodge decomposition or the Hodge filtration.
For any $p\in \mathcal{T}$ and $f^k=\dim F^k_p$ for any $0\leq k\leq n$,
We call a basis
\begin{align*}
\zeta=\{\zeta_0, \zeta_1, \cdots, \zeta_N, \cdots, \zeta_{f^{k+1}}, \cdots, \zeta_{f^k-1}, \cdots, \zeta_{f^2}, \cdots, \zeta_{f^1-0}, \zeta_{f^0-1}\}
\end{align*}
of $H^n(X, \mathbb{C})$ an \textit{adapted basis for the given filtration}
\begin{align*}
F^n\subseteq F^{n-1}\subseteq\cdots\subseteq F^0
\end{align*}
if it satisfies $F^{k}=\text{Span}_{\mathbb{C}}\{\zeta_0, \cdots, \zeta_{f^k-1}\}$ with $\text{dim}_{\mathbb{C}}F^{k}=f^k$.

\noindent The orthogonal group of the bilinear form $Q$ in the definition of
Hodge structure is a linear algebraic group, defined over $\mathbb{Q}$.
Let us simply denote $H_{\mathbb{C}}=H^n(X, \mathbb{C})$ and $H_{\mathbb{R}}=H^n(X, \mathbb{R})$.
The group of the $\mathbb{C}$-rational points is
\begin{align*}
G_{\mathbb{C}}=\{ g\in GL(H_{\mathbb{C}})|~ Q(gu, gv)=Q(u, v) \text{ for all } u, v\in H_{\mathbb{C}}\},
\end{align*}
which acts on $\check{D}$ transitively. The group of real points in $G_{\mathbb{C}}$ is
\begin{align*}
G_{\mathbb{R}}=\{ g\in GL(H_{\mathbb{R}})|~ Q(gu, gv)=Q(u, v) \text{ for all } u, v\in H_{\mathbb{R}}\},
\end{align*}
which acts transitively on $D$ as well.

\noindent Consider the period map $\Phi:\, \mathcal{T}\rightarrow D$.
Fix a point $p\in \mathcal{T}$ with the image $O:=\Phi(p)=\{F^n_p\subset \cdots\subset F^{0}_p\}\in D$.
The points $p\in \mathcal{T}$ and $O \in D$ may be referred as the base points or the reference points.
A linear transformation $g\in G_{\mathbb{C}}$ preserves the base point if and only if $gF^k_p=F^k_p$ for each $k$.
Thus it gives the identification
\begin{align*}
\check{D}\simeq G_\mathbb{C}/B\quad\text{with}\quad B=\{ g\in G_\mathbb{C}|~ gF^k_p=F^k_p, \text{ for any } k\}.
\end{align*}
Similarly, one obtains an analogous identification
\begin{align*}
D\simeq G_\mathbb{R}/V\hookrightarrow \check{D}\quad\text{with}\quad V=G_\mathbb{R}\cap B,
\end{align*}
where the embedding corresponds to the inclusion $
G_\mathbb{R}/V=G_{\mathbb{R}}/G_\mathbb{R}\cap B\subseteq G_\mathbb{C}/B.$
The Lie algebra $\mathfrak{g}$ of the complex Lie group $G_{\mathbb{C}}$ can be described as
\begin{align*}
\mathfrak{g}&=\{X\in \text{End}(H_\mathbb{C})|~ Q(Xu, v)+Q(u, Xv)=0, \text{ for all } u, v\in H_\mathbb{C}\}.
\end{align*}
It is a simple complex Lie algebra, which contains
$\mathfrak{g}_0=\{X\in \mathfrak{g}|~ XH_{\mathbb{R}}\subseteq H_\mathbb{R}\}$
as a real form, i.e. $\mathfrak{g}=\mathfrak{g}_0\oplus i \mathfrak{g}_0.$ With the inclusion $G_{\mathbb{R}}\subseteq G_{\mathbb{C}}$, $\mathfrak{g}_0$ becomes Lie algebra of $G_{\mathbb{R}}$. One observes that the reference Hodge structure $\{H^{k, n-k}_p\}_{k=0}^n$ of $H^n(M,{\mathbb{C}})$ induces a Hodge structure of weight zero on
$\text{End}(H^n(M, {\mathbb{C}})),$ namely,
\begin{align*}
\mathfrak{g}=\bigoplus_{k\in \mathbb{Z}} \mathfrak{g}^{k, -k}\quad\text{with}\quad\mathfrak{g}^{k, -k}=\{X\in \mathfrak{g}|XH_p^{r, n-r}\subseteq H_p^{r+k, n-r-k}\}.
\end{align*}
Since the Lie algebra $\mathfrak{b}$ of $B$ consists of those $X\in \mathfrak{g}$ that preserves the reference Hodge filtration $\{F_p^n\subset\cdots\subset F^0_p\}$, one thus has
\begin{align*}
 \mathfrak{b}=\bigoplus_{k\geq 0} \mathfrak{g}^{k, -k}.
\end{align*}
The Lie algebra $\mathfrak{v}_0$ of $V$ is
$\mathfrak{v}_0=\mathfrak{g}_0\cap \mathfrak{b}=\mathfrak{g}_0\cap \mathfrak{b}\cap\bar{\mathfrak{b}}=\mathfrak{g}_0\cap \mathfrak{g}^{0, 0}.$
With the above isomorphisms, the holomorphic tangent space of $\check{D}$ at the base point is naturally isomorphic to $\mathfrak{g}/\mathfrak{b}$.

\noindent Let us consider the nilpotent Lie subalgebra $\mathfrak{n}_+:=\oplus_{k\geq 1}\mathfrak{g}^{-k,k}$. Then one gets the holomorphic isomorphism $\mathfrak{g}/\mathfrak{b}\cong \mathfrak{n}_+$.
Since $D$ is an open set in $\check{D}$, we have the following relation:
\begin{align}\label{n+iso}
T^{1,0}_{O, h}D= T^{1, 0}_{O, h}\check{D}\cong\mathfrak{b}\oplus \mathfrak{g}^{-1, 1}/\mathfrak{b}\hookrightarrow \mathfrak{g}/\mathfrak{b}\cong \mathfrak{n}_+.
\end{align}
We define the unipotent group $N_+=\exp(\mathfrak{n}_+)$.

\begin{remark}\label{N+inD}
With a fixed base point, we can identify $N_+$ with its unipotent orbit in $\check{D}$ by identifying
an element $c\in N_+$ with $[c]=cB$ in $\check{D}$; that is, $N_+=N_+(\text{ base point })\cong N_+B/B\subseteq\check{D}.$
In particular, when the base point $O$ is in $D$, we have $N_+\cap D\subseteq D$.
We can also identify a point $\Phi(p)=\{ F^n_p\subseteq F^{n-1}_p\subseteq \cdots \subseteq F^{0}_p\}\in D$
 with any fixed adapted basis of the corresponding Hodge filtration, we have matrix representations of elements in the above Lie groups and Lie algebras. For example, elements in $N_+$ can be realized as nonsingular block upper triangular matrices with identity blocks in the diagonal; elements in $B$ can be realized as nonsingular block lower triangular matrices.
\end{remark}

\subsection{Quantum Correction and Calabi-Yau Threefolds} \label{hodgesymmetric}

For any point $p\in \mathcal{T}$, let $(X_p, L)$ be the corresponding fiber in the versal
family $\mathcal{X} \rightarrow \mathcal{T}$, which is a polarized and marked
Calabi--Yau manifold. The period map from $\mathcal{T}$ to $D$
 is defined by
assigning each point $p\in \mathcal{T}$ the Hodge structure on $X_p$, that is
\begin{align*}
\Phi: \mathcal{T} \rightarrow D, \quad\quad p\mapsto\Phi(p)=\{F^n(X_p)\subset\cdots\subset F^0(X_p)\}.
\end{align*}

\noindent In \cite{GS}, Griffiths and Schmid studied the Hodge metric over the period domain $D$.
In particular, this Hodge theory is a complete homogenous metric. Consider the period map on the
Teichm\"uller space $\Phi: \mathcal{T} \rightarrow D$. By local Torelli theorem for Calabi-Yau
manifolds, we know that the period map $\Phi$ is locally injective. Thus it follows from \cite{GS} that
the pull-back of the Hodge metric over $D$ by $\Phi$ on $\mathcal{T}$ is well-defined K\"ahler metrics.
We will call the pull-back metric the Hodge metric over Teichm\"uller space $\mathcal{T}$, denoted by $h$.

\begin{theorem}
Let $\mathcal{T}$ be the Teichm\"uller space of polarized and marked Calabi-Yau threefolds
and $\Phi: \mathcal{T} \rightarrow D$ be the period map.
Then the following conditions are equivalent:
\begin{enumerate}
\item There is no strong quantum correction
 at any point $p \in \mathcal{T}$.
\item With respect to the Hodge metric,
the image $ \Phi(\mathcal{T})$ is an open submanifold of
a globally Hermitian symmetric space $W$ of the same dimension as $\mathcal{T}$,
which is also a totally geodesic submanifold of
the period domain $D$.
\end{enumerate}
\end{theorem}

\noindent This theorem also implies that, under the assumption of no quantum correction at any point $p \in \mathcal{T}$,
the Teichm\"uller space $\mathcal{T}$ is a locally Hermitian symmetric space
and its image $\Phi(\mathcal{T})$ under the period map is a totally geodesic submanifold
of the period domain $D$, both with the natural Hodge metrics.
Moreover, assuming the global Torelli theorem in \cite{CGL13} for Calabi-Yau manifolds,
another consequence of this theorem is that the period map $\Phi$
embed the Teichm\"uller space $\mathcal{T}$ as a Zariski-open subset in the totally geodesic
submanifold $W$.

\begin{proof}
Fix $p \in \mathcal{T}$, let $X$ be the corresponding Calabi-Yau manifold in the versal family
and $\phi_1, \cdots, \phi_N \in \mathbb{H}^{0,1}(X, T^{1,0}X)$ be an orthonormal basis with respect to the Calabi-Yau metric.
If $\Omega$ is a nowhere vanishing holomorphic $(n,0)$-form over $X$ and
$\eta_i = \phi_i \lrcorner \Omega \in \mathbb{H}^{2,1}(X)$, then
$[\Omega], [\eta_1], \cdots , $ $[\eta_N], [\overline{\eta}_1], \cdots, [\overline{\eta}_N],[\overline{\Omega}]$
is a basis of $H^3(X)$ adapted the Hodge filtration of $X$.
Assume $t$ is the flat affine coordinate around $p \in \mathcal{T}$, then we have
$\phi_i = \kappa(\frac{\partial}{\partial t_i})$, where $\kappa: T_p^{1,0} \mathcal{T}
\rightarrow H^{0,1}(X , T^{1,0}X) \cong \mathbb{H}^{0,1}(X, T^{1,0}X)$ is the Kodaira-Spencer map.

\noindent And if we assume
\begin{eqnarray*}
\begin{bmatrix}[\phi_i \lrcorner \eta_1] \\ \vdots \\ [\phi_i \lrcorner \eta_N]\end{bmatrix}
= A_i \begin{bmatrix} [\overline{\eta}_1] \\ \vdots \\ [\overline{\eta}_N] \end{bmatrix},
\end{eqnarray*}
for some $N \times N$-matrix $A_i$, then, from the identity $[\phi_i \lrcorner \overline{\eta}_j] = \delta_{ij} [\overline{\Omega}]$,  we have
\begin{eqnarray}
\begin{bmatrix}[\phi_i \lrcorner \Omega]\\ [\phi_i \lrcorner \eta]
\\ [\phi_i \lrcorner \overline{\eta}] \\
 [\phi_i \lrcorner \overline{\Omega}] \end{bmatrix}=
 \begin{bmatrix} 0 & e_i & ~& ~\\
 ~ & 0_{N \times N}& A_i & ~\\
 ~& ~ & 0_{N \times N} & e_i^T \\
 ~& ~& ~& 0 \end{bmatrix} \begin{bmatrix}
 [\Omega] \\ [\eta]\\ [\overline{\eta}]\\ [\overline{\Omega}]\end{bmatrix}= E_i \begin{bmatrix}
[\Omega]\\ [\eta]\\ [\overline{\eta}] \\ [\Omega] \end{bmatrix},
\end{eqnarray}
where $e_i =(0, \cdots, 1, \cdots, 0)$,~~$[\eta]= \begin{bmatrix} [\eta_1], \cdots, [\eta_N] \end{bmatrix}^T $,~~
$[\overline{\eta}]= \begin{bmatrix} [\overline{\eta}_1], \cdots, [\overline{\eta}_N] \end{bmatrix}^T$,~~
$[\phi_i \lrcorner \eta] = \\ \begin{bmatrix}[\phi_i \lrcorner \eta_1], \cdots, [\phi_i \lrcorner \eta_N] \end{bmatrix}^T$
and  $[\phi_i \lrcorner \overline{\eta}]=\begin{bmatrix} [\phi_i \lrcorner \overline{\eta}_1], \cdots,
[\phi_i \lrcorner \overline{\eta}_N] \end{bmatrix}^T$.

\noindent Moreover, if we define $A(t) = \sum_{i=1}^N t_i A_i$, then it is easy to check that
\begin{eqnarray*}
&&\sum_{i,j=1}^N t_i t_j [\phi_i \lrcorner
\phi_j \lrcorner \Omega]
= \sum_{i=1}^N t_i (t_1, \cdots, t_N) \begin{bmatrix} [\phi_i \lrcorner \eta_1]\\
\vdots \\ [\phi_i \lrcorner \eta_N]
 \end{bmatrix}
=(t_1, \cdots, t_N) A(t)
\begin{bmatrix} [\overline{\eta}_1] \\ \vdots \\ [\overline{\eta}_N] \end{bmatrix}.
\end{eqnarray*}
And, from the identity $[\phi_i \lrcorner \overline{\eta}_j] = \delta_{ij} [\overline{\Omega}]$,
 we know that
\begin{eqnarray*}
\frac{1}{3!}\sum_{i,j,k=1}^N  t_i t_j t_k [\phi_i \lrcorner \phi_j \lrcorner \phi_k \lrcorner \Omega]
&=& \frac{1}{3!}\left( \sum_{i=1}^N t_i\phi_i \right) \lrcorner \left( \sum_{j,k=1}^N
t_j t_k [\phi_j \lrcorner \phi_k \lrcorner \Omega] \right) \\
&=&\frac{1}{3!} (t_1, \cdots, t_N) A(t) \left(\sum_{i=1}^N t_i \phi_i \right) \lrcorner
\begin{bmatrix} [\overline{\eta}_1]\\ \vdots \\ [\overline{\eta}_N]  \end{bmatrix}\\
&=& \frac{1}{3!} (t_1, \cdots, t_N) A(t) \begin{bmatrix} t_1 \\ \vdots \\ t_N \end{bmatrix}.
\end{eqnarray*}
Therefore, the classic canonical form
\begin{eqnarray*}
\Omega^{cc}(t)=\exp(\sum_{i=1}^N t_i \phi_i ) \lrcorner \Omega
= \Omega + \sum_{i=1}^N t_i \phi_i \lrcorner \Omega
+ \frac{1}{2!} \sum_{i,j=1}^N t_i t_j \phi_i \lrcorner \phi_j \lrcorner \Omega
+ \frac{1}{3!} \sum_{i,j,k=1}^N t_i t_j t_k \phi_i \lrcorner \phi_j \lrcorner \phi_k \lrcorner \Omega,
\end{eqnarray*}
satisfies
\begin{eqnarray}\label{classic}
[\Omega^{cc}(t)] =\begin{bmatrix} 1, (t_1, \cdots, t_N), \frac{1}{2!} (t_1, \cdots, t_N)A(t) , \frac{1}{3!} (t_1, \cdots, t_N)
A(t) (t_1, \cdots, t_N)^T \end{bmatrix}
\begin{bmatrix} [\Omega] \\ [\eta] \\ [\overline{\eta}] \\ [\overline{\Omega}] \end{bmatrix}.
\end{eqnarray}

\noindent On the other hand, let $O:= \Phi(p)$ be the base point or reference point. If we define
$E_i := \Phi_*(\frac{\partial}{\partial t_i}) \in \mathfrak{g}^{-1,1}$,
then $\mathfrak{a}= \mbox{span}_{\mathbb{C}} \{E_1, E_2, \cdots, E_N\}$ is an abelian
Lie subalgebra, see \cite[Lemma 5.5.1, page 173]{CMP}.
So we can define
\begin{eqnarray*}
W:= \exp (\sum_{i=1}^N \tau_i E_i) \cap D,
\end{eqnarray*}
which is a totally geodesic submanifold of $D$ with respect to the natural Hodge metric.
We consider the map $\rho : \mathbb{C}^N \rightarrow \breve{D}$ given by
\begin{eqnarray*}
\rho(\tau)= \exp(\sum_{i=1}^N \tau_i E_i)
\end{eqnarray*}
where $\tau=(\tau_1, \cdots, \tau_N)$ is the standard coordinates of $\mathbb{C}^N$.
Then the Hodge-Riemann bilinear relations define a bounded domain
$\beth \subset \mathbb{C}^N$, which is biholomorphic to $W$ via the map $\rho$. As the
$E_i's$ commute with each other, thus, for any $z,w \in \beth$, we have
\begin{eqnarray*}
&& \rho(z)(\rho(w))^{-1} = \exp(\sum_{i=1}^N z_i E_i)  \exp(- \sum_{i=1}^N w_i E_i)
= \exp(\sum_{i=1}^N (z_i -w_i)E_i)=I \\
&& \Longleftrightarrow z_i = w_i ~~~~\text{for}~~~~ 1\leq i \leq N,
\end{eqnarray*}
i.e., $\rho$ is one-to-one, which means that it defines a global coordinate $\tau$ over $W$.

\noindent $(1) \Rightarrow (2)$:
First we show that there exists a small coordinate chart $U_p$ of $p \in \mathcal{T}$ such that $\Phi(U_p) \subset W$.
 For any point $q \in U_p$ with flat affine coordinate $t$,
there is a unique upper-triangle matrix $\sigma(t) \in N_+$, i.e.,
an nonsingular upper triangular block matrices with identity blocks in the diagonal,
 such that
\begin{eqnarray}
\sigma(t) \begin{bmatrix}
[\Omega], [\eta], [\overline{\eta}],
 [\overline{\Omega}]
\end{bmatrix}^T
\end{eqnarray}
is a basis of $H^3(X_q)$ adapted to the Hodge filtration at $X_q$.

\noindent
Also, we know that $[\Omega^c(t)]=[e^{\Phi(t)} \lrcorner \Omega]$ is a basis of $F^3(X_q)$,
so $[\frac{\partial}{\partial t_i} \Omega^c(t)] \in F^{2}(X_q), 1 \leq i \leq N$   by Griffiths' transversality.
By the assumption that the there is no strong quantum correction at point $p \in \mathcal{T}$
and Formula \ref{classic}, we have
\begin{eqnarray*}
[\Omega^c(t)]&=&[\Omega^{cc}(t)]= \begin{bmatrix} 1, (t_1, \cdots, t_N), \frac{1}{2!} (t_1, \cdots, t_N)A(t) , \frac{1}{3!} (t_1, \cdots, t_N)
A(t) (t_1, \cdots, t_N)^T \end{bmatrix}
\begin{bmatrix} [\Omega] \\ [\eta] \\ [\overline{\eta}] \\ [\overline{\Omega}] \end{bmatrix}.
\end{eqnarray*}
And, for any $1 \leq i \leq N$, we have
\begin{eqnarray*}
[\frac{\partial \Omega^c(t)}{\partial t_i}] &=& [\phi_i \lrcorner \Omega]
+ \sum_{j=1}^N t_j [\phi_i \lrcorner \phi_j \lrcorner \Omega]
+ \frac{1}{2!} \sum_{j,k=1}^N t_j t_k [\phi_i \lrcorner \phi_j \lrcorner \phi_k \lrcorner \Omega]\\
&=& [\eta_i] + (t_1, \cdots, t_N) \begin{bmatrix} \phi_i \lrcorner \eta_1 \\ \vdots \\ \phi_i \lrcorner
\eta_N \end{bmatrix} + \frac{1}{2} (t_1, \cdots, t_N) A(t) \begin{bmatrix} [\phi_i \lrcorner \overline{\eta}_1] \\
\vdots \\ [\phi_i \lrcorner \overline{\eta}_N]
\end{bmatrix},
\end{eqnarray*}
which implies that
\begin{eqnarray*}
\begin{bmatrix}
[\frac{\partial \Omega^c(t)}{\partial t_1}] \\ \vdots \\ [\frac{\partial \Omega^c(t)}{\partial t_N}]
\end{bmatrix}
&=& [\eta]+ A(t)  [\overline{\eta}]+ \frac{1}{2!} (t_1, \cdots, t_N) A(t)~ [\overline{\Omega}]\\
&=& \begin{bmatrix} 0_{n \times 1}, I_{n \times n}, A(t) , \frac{1}{2} A^T(t)(t_1, \cdots , t_N)^T  \end{bmatrix}
\begin{bmatrix} [\Omega] \\ [\eta] \\ [\overline{\eta}] \\ [\overline{\Omega}] \end{bmatrix}.
\end{eqnarray*}

\noindent Thus $\{[\Omega^c(t)], [\frac{\partial \Omega^c(t)}{\partial t_1}] \cdots [\frac{\partial \Omega^c(t)}{\partial t_N}]\}$
is a basis of $F^2(X_q)$ as they are linearly independent, which implies that
the unique matrix $\sigma(t)\in N_+$ has the form
\begin{eqnarray*}
\sigma(t)= \begin{bmatrix}
I & (t_1, \cdots, t_N) & \frac{1}{2}(t_1, \cdots, t_N)A(t) &
\frac{1}{3!} (t_1, \cdots, t_N) A(t) (t_1, \cdots, t_N)^T \\
~& I & A(t) & \frac{1}{2} A^T(t) (t_1, \cdots, t_N)^T\\
~ & ~ & I & * \\
~ & ~ & ~ & I
\end{bmatrix},
\end{eqnarray*}
where $*$ represents an unknown column vector.
Meanwhile, we know $\sigma(t) \in G_{\mathbb{C}}$, i.e.
\begin{eqnarray*}
\sigma(T)^T \begin{bmatrix}
~&~&~& 1\\
~&~& - I_{N \times N}&~\\
~&I_{N \times N} &~&~\\
-1 &~&~&~
\end{bmatrix} \sigma(t)= \begin{bmatrix}
~&~&~& 1\\
~&~& - I_{N \times N}&~\\
~&I_{N \times N} &~&~\\
-1 &~&~&~
\end{bmatrix},
\end{eqnarray*}
which implies that
\begin{eqnarray*}
\sigma(t)= \begin{bmatrix} I & (t_1, \cdots, t_N) & \frac{1}{2}(t_1, \cdots, t_N)A(t) &
\frac{1}{3!} (t_1, \cdots, t_N) A(t) (t_1, \cdots, t_N)^T \\
~& I & A(t) & \frac{1}{2} A^T(t) (t_1, \cdots, t_N)^T\\
~ & ~ & I & (t_1, \cdots, t_N)^T \\
~ & ~ & ~ & I\end{bmatrix}.
\end{eqnarray*}
\noindent  Direct computation shows that
\begin{eqnarray}
\sigma(t)= \exp \left( \begin{bmatrix}
0 & (t_1, \cdots, t_N) & ~ &~\\
~&0_{N \times N}& A&~\\
~&~&0_{N \times N}& (t_1, \cdots, t_N)^T\\
~&~&~&0
 \end{bmatrix}\right)
 = \exp(\sum_{i=1}^N t_i E_i) \in W,
\end{eqnarray}
i.e., $\Phi(q) \in W$ for any point $q \in U_p$, which implies that $\Phi(U_p) \subset W$.

\noindent Take $q \in \mathcal{T}$ such that
$\Phi(q) \in \Phi(\mathcal{T}) \cap W$, then, the same argument for the point $q \in \mathcal{T}$,
there exists a neighborhood $U_q$ of $q$ such that $\Phi(U_q) \subset W$,
so $\Phi(U_q) \subset \Phi(\mathcal{T})\cap W$. By the local Torelli theorem
for Calabi-Yau manifolds, $\Phi(U_q) \subset \Phi(\mathcal{T})\cap W$ is an open
neighborhood of $\Phi(q)$ as $\mathcal{T}$ and $W$ has the same dimension,
thus $\Phi(\mathcal{T})\cap W \subset \Phi(\mathcal{T})$
is an open subset of $\Phi(\mathcal{T})$.
On the other hand, $W = \exp(\sum_{i=1}^N \tau_i E_i) \cap D \subset D$ is a closed
subset of the period domain $D$.
The closedness of
$W \subset D$ implies that $\Phi(\mathcal{T}) \cap W  \subset \Phi(\mathcal{T})$ is also a closed subset.
Therefore, $\Phi(\mathcal{T})\cap W = \Phi(\mathcal{T})$, i.e., $\Phi(\mathcal{T}) \subset W$
as $\Phi(\mathcal{T})$ is connected and not empty.

\noindent $(2) \Rightarrow (1)$: Fix $p \in \mathcal{T}$, we have
$\Phi(\mathcal{T}) \subset W = \exp (\sum_{i=1}^N \tau_i E_i) \cap D$
with $\tau \in \beth$.
If the flat affine coordinate of $q \in U_p$ is $t$
where $U_p$ is a small coordinate chart,
then we know the canonical family of $(3,0)$-classes is given by
\begin{eqnarray*}
[\Omega^c(t)]=[e^{\Phi(t)} \lrcorner \Omega] \in F^3(X_q).
\end{eqnarray*}
Moreover, from the fact that $\Phi(q)\in W$, there
exists $\tau=(\tau_1, \cdots, \tau_N) \in \beth$ such that
$$\Phi(q)= \exp(\sum_{i=1}^N \tau_i E_i),$$
i.e., the image of $q\in U_p$ under the period map $\Phi$ is
\begin{eqnarray*}
\Phi(q)&=&\exp(\sum_{i=1}^N \tau_i E_i) \begin{bmatrix} [\Omega]\\ [\eta] \\ [\overline{\eta}] \\ [\overline{\Omega}]] \end{bmatrix}\\
&=&\begin{bmatrix} 1 & (\tau_1, \cdots, \tau_N) & \frac{1}{2}(\tau_1, \cdots, \tau_N)A(\tau) &
\frac{1}{3!} (\tau_1, \cdots,\tau_N) A(\tau) (\tau_1, \cdots, \tau_N)^T \\
~& I_{N \times N} & A(\tau) & \frac{1}{2}A^T(\tau) (\tau_1, \cdots, \tau_N)^T\\
~ & ~ & I_{N \times N} & (\tau_1, \cdots, \tau_N)^T \\
~ & ~ & ~ & 1\end{bmatrix} \begin{bmatrix} [\Omega]\\ [\eta] \\ [\overline{\eta}] \\ [\overline{\Omega}]] \end{bmatrix},
\end{eqnarray*}
which is a basis adapted to the Hodge filtration over $X_q$. In particular, we know the first element in the basis
\begin{eqnarray*}
A_q(t) &=&(1, (\tau_1, \cdots, \tau_N), \frac{1}{2!} (\tau_1, \cdots, \tau_N) A(\tau), \frac{1}{3!} (\tau_1, \cdots,\tau_N)
A(\tau) (\tau_1, \cdots, \tau_N)^T) \begin{bmatrix} [\Omega]\\ [\eta] \\ [\overline{\eta}] \\ [\overline{\Omega}] \end{bmatrix}
\\ &\in& F^3 (X_q)= H^{3,0}(X_q).
\end{eqnarray*}
Thus, by the fact that
$H^{3,0}(X_q) \cong \mathbb{C}$ and $[\Omega^c(t)] \in H^{3,0}(X_q)$, there exists a constant
$\lambda \in \mathbb{C}$ such that
\begin{eqnarray*}
A_q(t) = \lambda [\Omega^c(t))].
\end{eqnarray*}
Also, we have
\begin{eqnarray*}
\mbox{Pr}_{H^{3,0}(X)}(A_q(t))= \mbox{Pr}_{H^{3,0}(X)}([\Omega^c(t)])=[\Omega],
\end{eqnarray*}
so $A_q(t) = [\Omega^c(t))] $. Thus
\begin{eqnarray}
\mbox{Pr}_{H^{2,1}(X)}(\Omega^c(t))= \mbox{Pr}_{H^{2,1}(X)}(A_q(t))= \sum_{i=1}^N \tau_i [\eta_i],
\end{eqnarray}
and,  from Corollary \ref{canonicalclass},
\beq
[\Omega^c(t)]=[\Omega]+ \sum^N_{i=1} t_i [\eta_i] + A(t),
\eeq
where $A(t) \in H^{1, 2}(X) \oplus H^{0,3}(X)$. Then we have
\begin{eqnarray*}
\sum_{i=1}^N t_i [\eta_i] = \sum_{i=1}^N \tau_i [\eta_i],
\end{eqnarray*}
which implies that $\tau_i = t_i,~ 1  \leq i \leq N$ as $\{[\eta_i]\}_{i=1}^N$
is a basis of $H^{2,1}(X)$. Therefore,
\begin{eqnarray*}
[\Omega^c(t)]= (1, (t_1, \cdots, t_N), \frac{1}{2!} (t_1, \cdots, t_N) A(t), \frac{1}{3!} (t_1, \cdots,t_N)
A(t) (t_1, \cdots, t_N)^T) \begin{bmatrix} [\Omega]\\ [\eta] \\ [\overline{\eta}] \\ [\overline{\Omega}] \end{bmatrix}
= [\Omega^{cc}(t)],
\end{eqnarray*}
by Formula \ref{classic}, i.e., there is no strong quantum correction at point $p \in \mathcal{T}$.
\end{proof}

\section{Compact Hyperk\"ahler Manifolds}\label{hkmanifold}
In Section \ref{preliminary}, we review some preliminary results about
Hyperk\"ahler manifolds and the period domain of weight $2$.
In Section \ref{local family}, we construct the canonical families of
$(2,0)$ and $(2n,0)$-classes by using the canonical families of smooth forms
$e^{\Phi(t)} \lrcorner \Omega^{2,0}$ and holomorphic forms $e^{\Phi(t)} \lrcorner \wedge^n \Omega^{2,0}$.
In Section $\ref{global family}$, we prove the expansions of the canonical families of $(2,0)$ and $(2n,0)$-classes
are actually globally defined over Teichm\"uller space $\mathcal{T}$.

\subsection{Preliminary Results}\label{preliminary}
In this section, we will review some preliminary results about Hyperk\"{a}hler
manifolds and the period domain.
We define Hyperk\"{a}hler manifolds as follows,
\begin{definition}
Let $X$ be a compact and simply-connected
K\"{a}hler manifold of complex dimension $2n\geq4$ such
that there exists a non-zero holomorphic non-degenerate $(2,0)$-form
$\Omega^{2,0}$ on $X$, unique up to a constant such that $\det(\Omega^{2,0})\neq0$
 at each point $x\in X$ and $H^{1}(X,\mathcal{O}_{\text{X}%
})=0.$ Then $X$ is called a Hyperk\"{a}hler manifold.
\end{definition}
\noindent The conditions on
the holomorphic $(2,0)$-form $\Omega^{2,0}$ imply that $\dim_{\mathbb{C}%
}H^{2}(X, \mathcal{O}_{\text{X}})=1.$
 A pair $(X,L)$ consisting of a Hyperk\"ahler manifold
 $X$ of complex dimension $2n$ with $2n\geq 4$ and an ample
line bundle $L$ over $X$ is called a polarized Hyperk\"ahler
manifold. By abuse of notation, the Chern class of $L$ will also be
denoted by $L$ and thus $L\in H^2(X,\mathbb{Z})$.
Let $\omega=\omega_g$ correspond to the Calabi-Yau metric in the class $L$, then
\begin{eqnarray*}
\mathbb{H}^{0,1}_L (X, T^{1,0}X)
=\{ \phi \in \mathbb{H}^{0,1}(X, T^{1,0}X) | [\phi \lrcorner \omega]=0\}
\end{eqnarray*}
And we know that if $\phi \in \mathbb{H}^{0,1}_L(X, T^{1,0}X)$ is harmonic, then
$\phi \lrcorner \omega$ is a harmonic $(0,2)$-form. Thus we have the identification
\begin{eqnarray*}
\mathbb{H}^{0,1}_L (X, T^{1,0}X)
=\{ \phi \in \mathbb{H}^{0,1}(X, T^{1,0}X) | \phi \lrcorner \omega=0\}
\end{eqnarray*}
Furthermore, the primitive cohomology groups satisfy:
\begin{eqnarray*}
H^{1,1}_{pr}(X) & \cong &\mathbb{H}^{1,1}_{pr}(X)= \{\eta \in H^{1,1}(X)| \eta \wedge \omega^{2n-1}=0\}\\
H^2_{pr}(X)&=& H^{2,0}(X) \oplus H^{1,1}_{pr}(X) \oplus H^{0,2}(X)
\end{eqnarray*}
The primitive cohomology group $H^2_{pr}(X)$ carry a nondegenerate bilinear form, the so-called Hodge
bilinear form
\begin{eqnarray}
Q(\alpha, \beta)=- \int_X \omega^{2n-2} \wedge \alpha \wedge \beta, ~~~~~~~~~\alpha, \beta \in H^2_{pr}(X),
\end{eqnarray}
which is evidently defined over $\mathbb{Q}$.

\noindent We consider the decreasing Hodge filtration $H^2_{pr}(M, \mathbb{C})=F^0 \supset F^1 \supset F^2$ with condition
\begin{eqnarray}\label{dim}
\mbox{dim}_{C} F^2=1, ~~~\mbox{dim}_{\mathbb{C}} F^1= b_2-2,~~~~~\mbox{dim}_{\mathbb{C}} F^0= b_2 -1.
\end{eqnarray}
Then the Hodge-Riemann relations becomes
\begin{eqnarray}\label{HR1}
Q(F^k, F^{3-k})=0,
\end{eqnarray}
\begin{eqnarray}\label{HR2}
Q(C v, \overline{v})>0  ~~~\text{if}~~v \neq 0,
\end{eqnarray}
where $C$ is the Weil operator given by $C v = (\sqrt{-1})^{2k-2} v$ for
$v \in H^{k, 2-k}_{pr}(M)= F^k \cap \overline{F^{2-k}}$.  The period domain
$D$ for polarized Hodge structures with data $\ref{dim}$ is the space of all such
Hodge filtrations
\begin{eqnarray*}
D = \{ F^2 \subset F^1 \subset F^0 = H^2_{pr}(X, \mathbb{C})| \ref{dim},~~ \ref{HR1}~~ \text{and} ~~\ref{HR2}~~ \text{hold}\}.
\end{eqnarray*}
The compact dual $\breve{D}$ of $D$ is
\begin{eqnarray*}
\breve{D }= \{ F^2 \subset F^1 \subset F^0 = H^2_{pr}(X, \mathbb{C})| \ref{dim}~~\text{and}~~ \ref{HR1}~~\text{hold}\}.
\end{eqnarray*}
The period domain $D \subseteq \breve{D}$ is an open subset.
We may identify the period space $D$ with the Grassmannian of positive
$2$-plans in $L^\perp$, and this gives us
\begin{eqnarray*}
D \cong SO(b_2-3, 2) / SO(2)\times SO(b_2-3).
\end{eqnarray*}
So the period domain $D$ is a global Hermitian symmetric space.

\noindent Let $\mathcal{T}= \mathcal{T}_L$ be the Teichm\"uller space of the polarized irreducible Hyperk\"ahler manifold $(X, L)$
which is a smooth complex manifold. The following result follows from  \cite{V09} or \cite{CGL13, CGL14},
\begin{theorem}
The period map
\begin{eqnarray*}
\Phi:\, \mathcal{T} \rightarrow D
\end{eqnarray*}
is injective.
\end{theorem}

\subsection{Local Family of $(2,0)$- and $(2n,0)$-Classes}\label{local family}

In this section, we derive the expansions of the canonical families of $(2,0)$ and $(2n,0)$-calsses
$[\mathbb{H}(e^{\Phi(t)}\lrcorner \Omega^{2,0})]$ and $[e^{\Phi(t)}\lrcorner \wedge^n \Omega^{2,0}]$,
where $\Omega^{2,0}$ is a nowhere vanishing holomorphic $(2,0)$-form over the Hyperk\"aher manifold $X$.
First we have the following Bochner's principle for
compact Ricci-flat manifolds:
\begin{proposition}\label{bochner}(Bochner's principle)
On a compact K\"ahler Ricci-flat manifold, any holomorphic tensor field
(covariant or contravariant) is parallel.
\end{proposition}
\noindent The proof rests on the following formula, which follows from a tedious
but straightforward computation \cite[p. 142]{BY}: if $\tau$ is any tensor field,
\begin{eqnarray*}
\Delta (\|\tau\|^2)= \|\nabla \tau \|^2.
\end{eqnarray*}
Therefore $\Delta(\|\tau\|^2)$ is nonnegative, hence $0$ since its mean value over $X$ is $0$
by the Stokes' formula. It follows that $\tau$ is parallel.

\noindent Then we
consider the canonical family of smooth $(2,0)$-forms
$\Omega^{c;2,0}(t)= e^{\Phi(t)}\lrcorner \Omega^{2,0}$
whose harmonic projection has the following expansion,

\begin{theorem} \label{2-0}
Fix $p \in \mathcal{T}$, let $(X,L)$ be the corresponding polarized Hyperk\"ahler manifold
and $\Omega^{2,0}$ be a nonzero holomorphic nondegenerate $(2,0)$-form
over $X$ and $\left\{\phi_{i}\right\}_{i=1}^N  $ be an orthonormal basis $\mathbb{H}^{0,1}_L\left(X,
T^{1,0}X\right)$ with respect to the Calabi-Yau metric. Then, in a neighborhood $U$ of $p$,  there exists a
locally canonical family of smooth $(2, 0)$-forms,
\[
 \Omega^{c;2,0}(t)= e^{\Phi(t)}\lrcorner \Omega^{2,0},
 \]
which defines a canonical family of $(2,0)$-classes
\begin{equation}
[\mathbb{H} ( \Omega^{c;2,0}(t))]  = [\Omega^{2,0}]+
{\displaystyle\sum\limits_{i=1}^{N}}
 [\phi_{i}\lrcorner\Omega^{2,0}] t_{i}+\frac{1}{2}%
{\displaystyle\sum\limits_{i=1}^{N}}
\left[   \phi_{i}\lrcorner\phi_{j}\lrcorner\Omega
^{2,0}  \right]  t_{i}t_{j}.\label{hk-0}%
\end{equation}%
\end{theorem}

\begin{proof}
Let us consider the canonical family of smooth $(2,0)$-forms
\begin{equation}%
\begin{array}
[c]{c}%
e^{\Phi(t)}\lrcorner\Omega^{2,0}=%
{\displaystyle\sum\limits_{i=1}^{N}}
\phi_{i}\lrcorner\Omega^{2,0}t^{i}+\frac{1}{2}%
{\displaystyle\sum\limits_{i,j=1}^{N}}
\left(  \phi_{i}\lrcorner\phi_{j}\lrcorner\Omega^{2,0}
+ \phi_{ij} \lrcorner \Omega^{2,0}\right)
t_{i}t_{j}\\
+{\displaystyle\sum\limits_{|I| \geq 3}}
\left(  \phi_I \lrcorner \Omega^{2,0}+ \suml_{J+K=I}
\phi_J \lrcorner \phi_K \lrcorner \Omega^{2,0}\right) .
\end{array}
\label{hk-3}%
\end{equation}
We claim that
\begin{claim}\label{claim1}
Suppose, for any multi-index $J, K$ with $|J| \geq 2$, the harmonic projection
$\mathbb{H}\left(  \phi_{J}\lrcorner\phi_{K}\lrcorner\Omega
^{2,0}\right)  = c \cdot \overline{\Omega^{2,0}}$ for some constant $c$, then we have
\begin{equation}%
\begin{array}
[c]{c}%
\mathbb{H}\left(   \phi_{J}\lrcorner\phi_{K}\lrcorner
\Omega^{2,0}  \wedge \wedge^n \Omega^{2,0}
 \wedge \wedge^{n-1} \bar{\Omega^{2,0}})\right)
=c \cdot \wedge^n \Omega^{2,0}\wedge \wedge^n \overline{\Omega^{2,0}}.
\end{array}
\label{hk-6}%
\end{equation}
\end{claim}
\bproof
By the Hodge decomposition, we have
\begin{eqnarray*}
\phi_J \lrcorner \phi_K \lrcorner \Omega^{2,0}= c \cdot \overline{\Omega^{2,0}}
+ d \alpha_1 + d^* \alpha_2.
\end{eqnarray*}

\noindent The proof bases on the following facts:
\begin{eqnarray}\label{covariant}
(\bar{\partial} \phi)_{A_p, \bar{\alpha B}_q} = \sum_{\alpha} \nabla_{\bar{\alpha}} \phi_{A_p, \bar{B}_q},~~~~
(\bar{\partial^*} \phi)_{A_p, \bar{B}_q}= (-1)^{p+1} \sum_{\alpha, \beta}
 g^{\bar{\beta} \alpha} \nabla_{\alpha} \phi_{A_p, \bar{\beta}\bar{B}_q},
\end{eqnarray}
and their conjugate, which can be found in \cite{Morrow-Kodaira71}.
From the formula
\begin{eqnarray}\label{wedge}
\nabla(\alpha \wedge \beta)= \nabla \alpha \wedge \beta + \alpha \wedge \nabla \beta,
\end{eqnarray}
and $\nabla \Omega^{2,0}=0$, which comes from the Bochner principal \ref{bochner}, we have
\begin{eqnarray}
\nabla (\alpha \wedge \wedge^n \Omega^{2,0} \wedge \wedge^{n-1} \overline{\Omega^{2,0}})
= \nabla \alpha \wedge \wedge^n \Omega^{2,0} \wedge \wedge^{n-1} \overline{\Omega^{2,0}}.
\end{eqnarray}
From the formulas \ref{covariant} and their conjugate, we have
\begin{eqnarray*}
d (\alpha_1 \wedge \wedge^n \Omega^{2,0} \wedge \wedge^{n-1} \overline{\Omega^{2,0}})
= d \alpha_1 \wedge \wedge^n \Omega^{2,0} \wedge \wedge^{n-1} \overline{\Omega^{2,0}}, \\
d^* (\alpha_2 \wedge \wedge^n \Omega^{2,0} \wedge \wedge^{n-1} \overline{\Omega^{2,0}})
= d^* \alpha_2 \wedge \wedge^n \Omega^{2,0} \wedge \wedge^{n-1} \overline{\Omega^{2,0}}.
\end{eqnarray*}
Thus we have
\begin{eqnarray*}
&&\phi_J \lrcorner \phi_K \lrcorner \Omega^{2,0} \wedge \wedge^n \Omega^{2,0} \wedge \wedge^{n-1} \overline{\Omega^{2,0}}\\
&=& c \cdot \wedge^n \Omega^{2,0} \wedge \wedge^n \overline{\Omega^{2,0}} + d(\alpha_1 \wedge
\wedge^n \Omega^{2,0} \wedge \wedge^{n-1} \overline{\Omega^{2,0}}) + d^*(\alpha_2 \wedge
\wedge^n \Omega^{2,0} \wedge \wedge^{n-1} \overline{\Omega^{2,0}}),
\end{eqnarray*}
which implies our claim.
\eproof
\noindent Direct computations show that
\begin{equation}
\left(  \phi_{J}\lrcorner\phi_{K}\lrcorner\Omega
^{2,0}\right)  \wedge \wedge^n \Omega^{2,0}=\left(  \phi_{J}
\lrcorner\Omega^{2,0}\right)  \wedge\left(  \phi_K
\lrcorner\Omega^{2,0}\right) \wedge \wedge^{n-1} \Omega^{2,0} \label{hk-5}%
\end{equation}
as a smooth $(2n, 2)$-form.  Therefore, for any multi-index $J, K$ with $|J| \geq 2$, we have
\begin{eqnarray*}
&& \int_X (\phi_{J}\lrcorner\phi_{K}\lrcorner
\Omega^{2,0})  \wedge\Omega^{2,0}
\wedge \wedge^{n-1 }\Omega^{2,0} \wedge \wedge^{n-1}\bar{\Omega^{2,0}}\\
&=& \int_X (\phi_{J}%
\lrcorner\Omega^{2,0})  \wedge (\phi_{K}%
\lrcorner\Omega^{2,0}) \wedge
\wedge^{n-1}\Omega^{2,0}\wedge \wedge^{n-1}\overline{\Omega^{2,0}}\\
&=& \frac{1}{n} \int_X (\phi_{J}
\lrcorner \wedge^n \Omega^{2,0})  \wedge (\phi_{K}
\lrcorner\Omega^{2,0}) \wedge
 \wedge^{n-1}\overline{\Omega^{2,0}}\\
 &=& \frac{1}{n} \int_X \partial \psi_{J} \wedge
 (\phi_{K}
\lrcorner\Omega^{2,0}) \wedge
 \wedge^{n-1}\overline{\Omega^{2,0}}\\
 &=& \frac{1}{n}\int_X \partial (\psi_{J} \wedge
 (\phi_{K}
\lrcorner\Omega^{2,0}) \wedge
 \wedge^{n-1}\overline{\Omega^{2,0}})~~~~~~~~(\text{as}~~
 \phi_K \lrcorner \Omega^{2,0} \wedge \wedge^{n-1} \overline{\Omega^{2,0}}~~\text{is d-closed}).\\
 &=& 0.
\end{eqnarray*}

\noindent On the other hand, by Claim \ref{claim1} and the Stokes' formula, we have
\begin{eqnarray*}
0=\int_X (\phi_{J}\lrcorner\phi_{K}\lrcorner
\Omega^{2,0})  \wedge\Omega^{2,0}
\wedge \wedge^{n-1 }\Omega^{2,0} \wedge \wedge^{n-1}\bar{\Omega^{2,0}}
= c \cdot \int_X \wedge^n \Omega^{2,0} \wedge \wedge^n \overline{\Omega^{2,0}}
\end{eqnarray*}

\noindent So we have $c=0$
i.e., $\mathbb{H}\left(  \left(  \phi_{J}\lrcorner\phi_{K}%
\lrcorner\Omega^{2,0}\right)  \right)  =0$ for any multiple-index
$J, K$ with $|J| \geq 2$.

\noindent Next, for the $(1,1)$-form $\phi_I \lrcorner \Omega^{2,0}$, we claim that
\begin{claim}
\begin{enumerate}
\item $\phi_i \lrcorner \Omega^{2,0}$ is harmonic for $1 \leq i \leq N$.
\item For any multi-index $I$ with $|I| \geq 2$,
$\phi_I \lrcorner \Omega^{2,0}$~ is $\partial$-exact,
which implies that $$\mathbb{H}(\phi_I \lrcorner \Omega^{2,0})=0.$$
\end{enumerate}
\end{claim}
\bproof
1. As $\phi_i$ is harmonic and $\Omega^{2,0}$ is parallel with
respect to Levi-Civita connection. So, from the formula \ref{covariant} and
\begin{eqnarray*}
\nabla(\phi_i \lrcorner \Omega^{2,0})= \nabla \phi_i \lrcorner \Omega^{2,0}
+ \phi_i \lrcorner \nabla \Omega^{2,0},
\end{eqnarray*}
we have
\begin{eqnarray*}
d(\phi_i \lrcorner \Omega^{2,0})=d\phi_i \lrcorner \Omega^{2,0}=0,\\
d^*(\phi_i \lrcorner \Omega^{2,0}) = d^* \phi_i \lrcorner \Omega^{2,0}=0,
\end{eqnarray*}
i.e., $\phi_i \lrcorner \Omega^{2,0}$ is harmonic for $1 \leq i \leq N$.

\noindent 2. As $\Omega^{2,0}$ is a nowhere vanishing holomorphic $(2,0)$-form, so we can define
$\Omega^{* 2,0} \in \Gamma(X, \wedge^2 \Theta)$ by requiring
$\langle \Omega^{2,0}, \Omega^{* 2,0} \rangle=1$ pointwise on $X$.
Actually, in a coordinate chart $\{z_1, z_2, \cdots, z_{2n}\}$, we can assume
\begin{eqnarray*}
 \Omega^{2,0} &=& \sum_{i,j=1}^{2n} a_{ij} dz_i \wedge dz_j ~~~\text{with}~~~
 a_{ij}=-a_{ji} \\
 \Omega^{*2, 0} &=& \sum_{i,j=1}^{2n} b_{ij} \frac{\partial}{\partial z_i} \wedge
 \frac{\partial}{\partial z_j}~~~ \text{with}~~~ b_{ij}=-b_{ji}.
\end{eqnarray*}
Then, if we define matrices  $A= (a_{ij})$ and $B = (b_{ij})$, then $\det(A)\neq 0$ and
\begin{eqnarray*}
\langle \Omega^{2,0}, \Omega^{*2,0} \rangle = \sum_{i,j=1}^{2n} a_{ij} b_{ij}
= \tr(A B^{T})=1,
\end{eqnarray*}
so, locally, the matrix $B$ can be defined by
\begin{eqnarray*}
B = \frac{1}{2n} (A^{-1})^{T}.
\end{eqnarray*}
And it is easy to check that this definition is independent of the local coordinates and
$\nabla \Omega^{*2,0}=0$ by the Bochner's principle \ref{bochner}. Then, by Theorem \ref{BTT}, we have
\begin{eqnarray*}
\phi_{I}
\lrcorner \wedge^n \Omega^{2,0}=\p_{X} \psi_{I}, ~~~~~ |I| \geq 2,
\end{eqnarray*}
which implies that
\begin{eqnarray*}
\phi_I \lrcorner \Omega^{2,0}= \wedge^{n-1} \Omega^{* 2,0}
 \lrcorner (\phi_{I}\lrcorner \wedge^n \Omega^{2,0})
= \wedge^{n-1} \Omega^{* 2,0} \lrcorner \partial \psi_I
= \partial (\wedge^{n-1} \Omega^{* 2,0} \lrcorner \psi_I),
\end{eqnarray*}
by formulas \ref{covariant} and \ref{wedge}.
\eproof

\noindent Thus, the harmonic projection of the family of $(2,0)$-forms $e^{\Phi(t)}%
\lrcorner\Omega^{2,0}$ is given by
\begin{eqnarray*}
&&\mathbb{H}\left(  e^{\Phi(t)}\lrcorner\Omega^{2,0}
\right) = \Omega^{2,0}+
{\displaystyle\sum\limits_{i=1}^{N}}
 \phi_{i}\lrcorner\Omega^{2,0}  t_{i}+\frac{1}{2}%
{\displaystyle\sum\limits_{i=1}^{N}}
\left(   \phi_{i}\lrcorner\phi_{j}\lrcorner\Omega
^{2,0} \right)  t_{i}t_{j}.
\end{eqnarray*}
\label{hk-7}%
Theorem \ref{2-0} is proved.
\end{proof}

\begin{corollary}\label{c-0}
Fix $p \in \mathcal{T}$, let $(X,L)$ be the corresponding polarized Hyperk\"ahler manifold
and $\Omega^{2,0}$ be a non-zero holomorphic non-degenerate $(2,0)$-form
over $X$ and $\left\{\phi_{i}\right\}_{i=1}^N  $ be an orthonormal basis $\mathbb{H}^{0,1}_L\left(X,
T^{1,0}X\right)$ with respect to the Calabi-Yau metric. Then, in a neighborhood $U$ of $p$,  there exists
a canonical family of holomorphic $(2n,0)$-forms,

$$ \Omega^{c}(t)  = e^{\Phi(t)} \lrcorner \wedge^n \Omega^{2,0} $$
 which defines a canonical family of $(2n,0)$-classes
\begin{equation}
\left[  \Omega^{c}(t)\right]  =\left[  \wedge^n \Omega^{2,0}\right]  +
{\displaystyle\sum\limits_{i=1}^{N}}
\left[  \phi_{i}\lrcorner  \wedge^n \Omega^{2,0}\right] t_{i}+\frac
{1}{k!}
{\displaystyle\sum\limits_{k=1}^{2n}}
\left(
{\displaystyle\sum\limits_{1\leq i_{1}\leq...\leq i_{k}\leq N}}
\left[  \phi_{i_{1}}\lrcorner...\lrcorner\phi_{i_{k}}\lrcorner\Omega
\right]  t_{i_1} t_{i_2} \cdots t_{i_k} \right). \label{n-0}
\end{equation}
In particular, the expansions imply that the Teichm\"uller space $\mathcal{T}$
is a locally Hermitian symmetric space.
\end{corollary}

\begin{proof} From the Proposition \ref{canonicalform}, we have the harmonic projection
\begin{eqnarray*}
\mathbb{H}(e^{\Phi(t)} \lrcorner \wedge^n \Omega^{2,0}) \in \mathbb{H}^{2n,0}(X_t).
\end{eqnarray*}
And from Theorem \ref{2-0}, we know that
\begin{eqnarray*}
\mathbb{H}(e^{\Phi(t)}\lrcorner \Omega^{2,0}) \in \mathbb{H}^{2,0} (X_t),
\end{eqnarray*}
therefore, we have
\begin{eqnarray*}
\mathbb{H}[\wedge^n \mathbb{H}(e^{\Phi(t)}\lrcorner \Omega^{2,0})] \in \mathbb{H}^{2n,0}(X_t).
\end{eqnarray*}
Because of $\dim_{\mathbb{C}} \mathbb{H}^{2n,0}(X_t)=1$, there exists $\lambda \in \mathbb{C}$ such that
\begin{eqnarray}
\mathbb{H}(e^{\Phi(t)} \lrcorner \wedge^n \Omega^{2,0})
 = \lambda \mathbb{H}[\wedge^n \mathbb{H}(e^{\Phi(t)}\lrcorner \Omega^{2,0})].
\end{eqnarray}
On the other hand, we have
\begin{eqnarray}
\mbox{Pr}_{\mathbb{H}^{2n,0}(X)} \left(\mathbb{H}(e^{\Phi(t)} \lrcorner \wedge^n \Omega^{2,0})\right)
=\mbox{Pr}_{\mathbb{H}^{2n,0}(X)}\left(\mathbb{H}[\wedge^n \mathbb{H}(e^{\Phi(t)}\lrcorner \Omega^{2,0})]\right)
= \wedge^n \Omega^{2,0}.
\end{eqnarray}
Thus $\lambda=1$, i.e.,
\begin{eqnarray*}
[\Omega^c(t)] &=&
[e^{\Phi(t)}\lrcorner \wedge^n \Omega^{2,0}]
= \left[\wedge^n \left(\Omega^{2,0}+
{\displaystyle\sum\limits_{i=1}^{N}}
 \phi_{i}\lrcorner\Omega^{2,0}  t_{i}+\frac{1}{2}%
{\displaystyle\sum\limits_{i=1}^{N}}
\left(   \phi_{i}\lrcorner\phi_{j}\lrcorner\Omega
^{2,0} \right)  t_{i}t_{j}\right)\right]\\
&=&\left[  \wedge^n  \Omega^{2,0}\right]  +
{\displaystyle\sum\limits_{i=1}^{N}}
\left[  \phi_{i}\lrcorner  \wedge^n \Omega^{2,0}\right] t_{i}+\frac
{1}{k!}
{\displaystyle\sum\limits_{k=1}^{2n}}
\left(
{\displaystyle\sum\limits_{0\leq i_{1}\leq...\leq i_{k}}}
\left[  \phi_{i_{1}}\lrcorner...\lrcorner\phi_{i_{k}}\lrcorner\Omega
\right]  t_{i_1} t_{i_2} \cdots t_{i_k} \right),
\end{eqnarray*}
which is a polynomial in terms of the flat affine coordinate $t=(t_1, \cdots, t_N)$.
Thus, by the Corollary \ref{corfinite}, the Teichm\"uller
space $\mathcal{T}$ of polarized Hyperk\"ahler manifolds is a locally Hermitian symmetric space.
\end{proof}

\subsection{Global Family of (2,0)- and (2n,0)-Classes}\label{global family}

In this section, we will show that the flat affine coordinate $t$
is globally defined over the Teichm\"uller space $\mathcal{T}$,
so do the expansions of the canonical families of $(2,0)$ and $(2n,0)$-classes.

\noindent Fix $p \in \mathcal{T}$, let $(X, L)$ be the corresponding polarized
Hyperk\"aher manifold.
Fix a nowhere vanishing holomorphic $(2,0)$-form $\Omega^{2,0}$ over $X$ and a basis of harmonic form $\eta_1, \cdots, \eta_N
\in \mathbb{H}^{1,1}_{pr}(X)$. Then $[\Omega^{2,0}], [\eta_1], \cdots, [\eta_N], [\overline{\Omega^{2,0}}]$
is a basis of $H^2_{pr}(X, \mathbb{C})$. We normalize this basis such that
$Q([\Omega^{2,0}], [\overline{\Omega^{2,0}}])=-1$ and $Q([\eta_i], [\eta_j])= \delta_{ij}$.
Moreover, we can pick up a harmonic basis of $\mathbb{H}^{0,1}_L(X, T^{1,0}X)$, denote by $\phi_1, \cdots, \phi_N$,
such that $\phi_i \lrcorner \Omega^{2,0}= \eta_i$.

\noindent Using the point $O= \Phi(p)$ as the base point or reference pint, we can parameterize the period domain $D$.
Let $e_i= (0, \cdots, 1, \cdots, 0)$ with $1 \leq i \leq N$ be the standard basis of $\mathbb{C}^N$. Here we view
$e_i$ as a row vector, then we can define
\begin{eqnarray*}
E_i =
\begin{bmatrix}
0 & e_i & 0\\
0 &0 & e_i^T\\
0&0&0
\end{bmatrix} \in \mathfrak{g}^{-1,1}.
\end{eqnarray*}
It follows that $E_i E_j=0$ if $i \neq j$ and
\begin{eqnarray*}
E_i^2= \begin{bmatrix}0&0&1\\0&0&0\\0&0&0  \end{bmatrix},
\end{eqnarray*}
which implies that
\begin{eqnarray*}
\exp(\sum_{i=1}^N t_i E_i)= \begin{bmatrix} 1 & (\tau_1, \cdots, \tau_N) & \frac{1}{2} \sum_{i=1}^N \tau_i^2\\
0 & I_{N \times N} & (\tau_1, \cdots, \tau_N)^T\\
0 & 0 &1 \end{bmatrix}.
\end{eqnarray*}
Let $\beth \subset \mathbb{C}^N$ be the domain enclosed by the real hypersurface
\begin{eqnarray*}
1 - \sum_{i=1}^N |\tau_i|^2 + \frac{1}{4} |\sum_{i=1}^N \tau_i^2|=0.
\end{eqnarray*}
Let $\tau=(\tau_1, \cdots, \tau_N)$ be the standard coordinates on $\mathbb{C}^N$, then the map
$\rho: \beth \rightarrow D$ given by $\rho(t)= \exp(\sum_{i=1}^N \tau_i E_i)$
is a biholormorphic map. This is the Harish-Chandra realization \cite{HC} of the period domain $D$.
Moreover, from the the global Torelli theorem (cf. \cite{V09} and \cite{CGL13, CGL14}) for Hyperk\"ahler
manifolds, we know that the map
\begin{eqnarray}
\rho^{-1} \circ \Phi : \mathcal{T} \rightarrow \beth,
\end{eqnarray}
is an injective map. So the coordinate $\tau= (\tau_1, \cdots, \tau_N)$ are global coordinates on
the Teichm\"uller space $\mathcal{T}$. We call it the Harish-Chandra coordinate.
We know that, in a neighborhood of $p \in \mathcal{T}$, there is another flat
affine coordinate $t$. Actually, these two coordinates coincide,  and we have the following theorem:

\begin{theorem}
Fix $p \in \mathcal{T}$, in a neighborhood $U$ of $p $, the global Harish-Chandra coordinate $\tau$ coincide with the
flat affine coordinate $t$. So the flat affine coordinate $t$ is globally defined and the 
expansions of the canonical families of
cohomology classes \ref{hk-0} and \ref{n-0}, i.e.,
\begin{eqnarray*}
&&[\Omega^{2,0}] +
{\displaystyle\sum\limits_{i=1}^{N}}
 [\phi_{i}\lrcorner\Omega^{2,0}] t_{i}+\frac{1}{2}%
{\displaystyle\sum\limits_{i=1}^{N}}
[  \phi_{i}\lrcorner\phi_{j}\lrcorner\Omega
^{2,0}   ] t_{i}t_{j} \in H^{2,0}(X_t),\\
&&\left[  \wedge^n \Omega^{2,0}\right] +
{\displaystyle\sum\limits_{i=1}^{N}}
\left[  \phi_{i}\lrcorner  \wedge^n \Omega^{2,0}\right] t_{i}\\
&&+\frac{1}{k!}
{\displaystyle\sum\limits_{k=1}^{2n}}
\left({\displaystyle\sum\limits_{1\leq i_{1}\leq...\leq i_{k}\leq N}}
[  \phi_{i_{1}}\lrcorner...\lrcorner\phi_{i_{k}}\lrcorner \wedge^n\Omega^{2,0}
]  t_{i_1} t_{i_2} \cdots t_{i_k} \right) \in H^{2n,0}(X_t).
\end{eqnarray*}
are globally defined over the Teichm\"uller space $\mathcal{T}$.
\begin{proof}
Let $t$ be its flat affine coordinate of $q \in U_p$ where $U_p$ is a small coordinate and $\tau$
be its Harish-Chandra coordinate. We only need to show that $t= \tau$. From the definition of
$\tau$, we know that
\begin{eqnarray*}
\exp(\sum_{i=1}^N \tau_i E_i) \begin{bmatrix} [\Omega^{2,0}]\\
[\eta_1]\\ \vdots \\ [\eta_N]\\ [\overline{\Omega^{2,0}}] \end{bmatrix}
=\begin{bmatrix} 1 & (\tau_1, \cdots, \tau_N) & \frac{1}{2} \sum_{i=1}^N \tau_i^2\\
0 & I_{N \times N} & (\tau_1, \cdots, \tau_N)^T\\
0 & 0 &1 \end{bmatrix} \begin{bmatrix} [\Omega^{2,0}]\\
[\eta_1]\\  \vdots \\ [\eta_N]\\ [\overline{\Omega^{2,0}}] \end{bmatrix},
\end{eqnarray*}
is a basis of $H^2(X_q)$ adapted to the Hodge filtration of $X_q$.
Consider the first element in the basis, we have
\begin{eqnarray*}
B_q (t) =[\Omega^{2,0}]+ \sum_{i=1}^N \tau_i [\eta_i] + \frac{1}{2} \sum_{i=1}^N\tau_i^2 [\overline{\Omega^{2,0}}]
\in H^{2,0}(X_q).
\end{eqnarray*}
On the other hand, from Theorem \ref{2-0}, we know that
\begin{eqnarray*}
[\mathbb{H} ( \Omega^{c;2,0}(t))]  &=& [\Omega^{2,0}]+
{\displaystyle\sum\limits_{i=1}^{N}}
[\phi_{i}\lrcorner\Omega^{2,0}] t_{i}+ \frac{1}{2}%
{\displaystyle\sum\limits_{i=1}^{N}}
\left[   \phi_{i}\lrcorner\phi_{j}\lrcorner\Omega
^{2,0}  \right]  t_{i}t_{j} \\
&=& [\Omega^{2,0}]+ \sum_{i=1}^N [\eta_i]t_i + \frac{1}{2} \sum_{i=1}^N t_i^2 [\bar{\Omega^{2,0}}]
\in H^{2,0}(X_q),
\end{eqnarray*}%
in the flat affine coordinate $t$.  Thus, by the fact that
$H^{2,0}(X_q) \cong \mathbb{C}$, there exists a constant
$\lambda \in \mathbb{C}$ such that
\begin{eqnarray*}
B_q(t) = \lambda [\mathbb{H} ( \Omega^{c;2,0}(t))].
\end{eqnarray*}
Moreover, we have
\begin{eqnarray*}
[\Omega^{2,0}]=\mbox{Pr}_{H^{2,0}(X)}(B_q(t))
= \lambda \mbox{Pr}_{H^{2,0}(X)}([\mathbb{H}(\Omega^{c; 2,0}(t)0])= \lambda [\Omega^{2,0}],
\end{eqnarray*}
so $\lambda =1$, i.e., $B_q(t) = [\mathbb{H}(\Omega^{c;2,0}(t))] $. Thus we know $t_i = \tau_i , ~1 \leq i \leq N$,
as $[\Omega^{2,0}], [\eta_1], \cdots, [\eta_N],$ $ [\overline{\Omega^{2,0}}]$ is a basis of $H^2(X)$.
\end{proof}
\end{theorem}

\end{document}